\documentclass[11pt, reqno]{amsart}

\usepackage{amsmath}
\allowdisplaybreaks
\usepackage{amsfonts}
\usepackage{amssymb}
\usepackage[all]{xy}           

\usepackage{bbding}
\usepackage{txfonts}
\usepackage{amscd}

\usepackage[shortlabels]{enumitem}
\usepackage{ifpdf}
\ifpdf
  \usepackage[colorlinks,final,backref=page,hyperindex]{hyperref}
\else
  \usepackage[colorlinks,final,backref=page,hyperindex,hypertex]{hyperref}
\fi
\usepackage{tikz}
\usepackage[active]{srcltx}

\makeatletter

\newtheorem{df}{Definition}[section]
\newtheorem{thm}{Theorem}[section]
\newtheorem{cor}{Corollary}[section]
\newtheorem{rem}{Remark}[section]

\newtheorem{prop}{Proposition}[section]
\newtheorem{exa}{Example}[section]
\newtheorem{lem}{Lemma}[section]

\numberwithin{equation}{section}

\begin{document}

\date{}

\title[Some results on Hom-Novikov superalgebras and Hom-anti-pre-Novikov superalgebras]
{Some results on Hom-Novikov superalgebras and Hom-anti-pre-Novikov superalgebras}

\author{W Ben Abdelhafidh}
\address{University of Sfax, Faculty of Sciences of Sfax,   BP 1171, 3000 Sfax, Tunisia}
\email{wiembenabdelhafidh@gmail.com}


\begin{abstract}

This paper investigates the structural properties, operator representations, and algebraic duality of Hom-Novikov superalgebras and to introduce Hom-anti-pre-Novikov superalgebras. We establish a formal equivalence between Hom-Novikov superalgebras and admissible Hom-Novikov superalgebras via $q$-superalgebra extensions, proving that a Hom-pre-Lie superalgebra admits a Hom-Novikov structure if and only if its associated $2$-superalgebra is admissible. Introducing $q$-admissible pairs $(\mathcal P, \mathcal Q)$ on super-commutative Hom-associative superalgebras, we detail systematic constructions of Hom-Novikov superalgebras and their admissibility classes. Furthermore, we characterize anti-super-$\mathcal{O}$-operators on Hom-Lie superalgebras, showing that invertible operators induce compatible Hom-anti-pre-Lie structures and extend directly to admissible setups. Extending these frameworks to Poisson-type structures, we demonstrate that the sub-adjacent Hom-Lie superalgebra of a Hom-anti-pre-Lie Poisson superalgebra yields a transposed Hom-Poisson superalgebra. Finally, we establish the fundamental connection between Hom-anti-pre-Novikov superalgebras and invertible anti-super-$\mathcal{O}$-operators.
\end{abstract}



\maketitle


\tableofcontents

\allowdisplaybreaks

\section{Introduction}

The theory of Novikov algebras originated from the pioneering works of Balinskii and Novikov \cite{Balinskii1985} and Gel'fand and Dorfman \cite{Gel} in their study of Hamiltonian operators, Poisson brackets of hydrodynamic type, and the formal variational calculus. A Novikov algebra is a vector space $\mathcal{A}$ equipped with a bilinear map $\star$ satisfying, for all $x,y,z\in\mathcal{A}$,
\begin{align}
(x\star y)\star z-x\star(y\star z)
&=(y\star x)\star z-y\star(x\star z),
\label{def-Nov-alg1}\\
(x\star y)\star z&=(x\star z)\star y.\label{def-Nov-alg2}
\end{align}
The first identity is the left-symmetry (or pre-Lie) identity, while the second is the right-commutativity identity. Since then, Novikov algebras have become one of the fundamental classes of nonassociative algebras owing to their deep connections with Lie algebras, left-symmetric (pre-Lie) algebras, integrable systems, differential geometry, vertex algebras, and mathematical physics. In particular, the commutator of a Novikov algebra naturally determines a Lie algebra, while every Novikov algebra is a left-symmetric algebra satisfying the additional right-commutativity identity. These remarkable structural properties have stimulated extensive investigations into their representations, cohomology, deformations, classifications, and applications \cite{Osborn1992,Xu1996,Bai2001}.

A major motivation for the development of Novikov algebras lies in their close relationship with Hamiltonian mechanics. Gel'fand and Dorfman \cite{Gel} showed that Novikov algebras provide the algebraic framework underlying Hamiltonian operators of hydrodynamic type. Subsequently, Xu \cite{Xu1996} established a systematic construction of Novikov algebras from commutative associative algebras endowed with derivations, thereby revealing a rich source of examples. Later, Bai and Meng further investigated realization problems and the classification of low-dimensional Novikov algebras, considerably enriching the structural theory and providing effective construction methods \cite{Bai2001,Bai-Meng}. Beyond their intrinsic algebraic interest, Novikov algebras have found numerous applications in Poisson geometry, conformal algebras, integrable hierarchies, operad theory, and related areas of mathematical physics.

The $\mathbb Z_2$-graded counterpart of Novikov algebras, namely Novikov superalgebras, was introduced to describe supersymmetric generalizations of Hamiltonian structures and quadratic conformal superalgebras. While preserving many of the fundamental properties of classical Novikov algebras, Novikov superalgebras exhibit new phenomena arising from the presence of the underlying grading. Their algebraic structure, classification, representation theory, and cohomological aspects have attracted increasing attention in recent years \cite{Kang-Chen,Ni-Chen}. In particular, the classification of low-dimensional Novikov superalgebras together with the study of quadratic Novikov superalgebras equipped with associative supersymmetric bilinear forms has revealed profound relationships among Novikov superalgebras, Lie superalgebras, and conformal superalgebras \cite{Kang-Chen,Ni-Chen}.

The emergence of Hom-type algebras has opened a broad new direction in the study of deformations and generalizations of classical algebraic structures. Hom-Lie algebras were first introduced by Hartwig, Larsson and Silvestrov in their investigation of $\sigma$-deformations of Witt and Virasoro algebras \cite{Hartwig-Larsson-Silvestrov}. Shortly thereafter, Makhlouf and Silvestrov established the general theory of Hom-associative and Hom-Lie algebras, providing a unified framework in which classical algebraic identities are twisted by linear self-maps \cite{Makhlouf-Silvestrov}. Since then, Hom-type structures have been successfully extended to numerous algebraic systems, including associative, Lie, pre-Lie, Poisson, Leibniz, dendriform, Zinbiel, and many related algebras \cite{D-Yau,Makhlouf-Silvestrov1}.

Motivated by these developments, Yau introduced Hom-Novikov algebras as Hom-type analogues of Novikov algebras and proved that every Novikov algebra together with an algebra endomorphism naturally induces a Hom-Novikov algebra \cite{D-Yau}. More precisely, a Hom-Novikov algebra is a Hom-pre-Lie superalgebra $(\mathcal A,\star,\alpha)$ such that
\begin{equation}\label{def-Hom-Nov-alg}
(x\star y)\star \alpha(z)=(x\star z)\star \alpha(y), \forall x,y,z\in\mathcal A,
\end{equation}

in the super setting, with the signs omitted in the ordinary (non-graded) case. Yau further established several construction methods based on Hom-associative algebras equipped with derivations, demonstrating that Hom-Novikov algebras retain many of the characteristic features of their classical counterparts while exhibiting richer deformation phenomena. Since then, Hom-Novikov algebras have been investigated from various perspectives, including Hom-Novikov-Poisson algebras, quadratic Hom-Novikov algebras, representations, cohomology, and deformation theory \cite{D-Yau1,Yuan-You}.

Another significant development in recent years is the introduction of admissible Novikov algebras. Motivated by the theory of anti-pre-Lie algebras and commutative $2$-cocycles on Lie algebras, Liu and Bai established the notion of admissible Novikov algebras and proved a one-to-one correspondence between Novikov algebras and admissible Novikov algebras through the theory of $q$-algebras \cite{Liu-Bai}. They also provided explicit constructions from commutative associative algebras endowed with derivations, offering a new perspective on the classical theory of Novikov algebras. More recently, these ideas have been extended to the super setting by introducing admissible Novikov superalgebras, which establish deep connections among Novikov superalgebras, anti-pre-Lie superalgebras, Lie superalgebras, and commutative super $2$-cocycles \cite{Chen-Liu-Chen}. These results not only generalize the correspondence between Novikov and admissible Novikov algebras but also furnish effective methods for constructing new superalgebraic structures and exploring their geometric and algebraic properties.

Closely related to these developments is the theory of anti-pre-Novikov (super)algebras, which has recently emerged as a natural nonassociative framework associated with symmetric quasi-Frobenius Novikov algebras and nondegenerate commutative $2$-cocycles \cite{Liu-Bai,Sun-Zeng}. Historically rooted in the study of Hamiltonian operators, variational calculus, and Poisson brackets of hydrodynamic type \cite{Chen-Chen-Ding}, anti-pre-Novikov algebras arise from suitable decompositions of Novikov structures through negative left and right multiplication representations \cite{Sun-Zeng}. This construction parallels the theories of anti-pre-Lie and anti-dendriform algebras, providing a unified mechanism for splitting algebraic identities and revealing deeper structural relationships \cite{Gao-Liu-Bai,Liu-Bai}. In the $\mathbb Z_2$-graded setting, these ideas naturally extend to anti-pre-Novikov superalgebras, admissible Novikov superalgebras, and their interactions with Lie superalgebras and quadratic conformal structures \cite{Chen-Liu-Chen}. Furthermore, the study of double constructions, anti-Rota-Baxter operators, and the anti-pre-Novikov Yang-Baxter equation (APN-YBE) has established important connections with factorizable and quasitriangular bialgebras, highlighting the rich interplay between anti-pre-Novikov (super)algebras and contemporary nonassociative algebraic structures \cite{Gao-Guo-Han-Zhang,Li-Hong,Sun-Zeng}. Motivated by these developments, we introduce in this paper the notion of a Hom-anti-pre-Novikov superalgebra, which constitutes a Hom-type extension of anti-pre-Novikov superalgebras obtained by incorporating an even twisting map into their defining identities. We systematically investigate the fundamental algebraic properties of these structures and develop several structural constructions, thereby extending the theory of anti-pre-Novikov superalgebras to the broader framework of Hom-algebraic structures.\\

The paper is organized as follows. In Section 2, we recall the basic notions of Hom-anti-pre-Lie superalgebras and anti-super-$\mathcal O$-operators, and we establish several structural results concerning strong operators and compatible Hom-Lie structures. Section 3 is devoted to Hom-Novikov superalgebras and admissible Hom-Novikov superalgebras, including their correspondence, operator interpretation, constructions from super-commutative Hom-associative superalgebras, and the associated Poisson-type structures. In Section 4, we introduce Hom-anti-pre-Novikov superalgebras, study their representations, and develop the theory of anti-super-$\mathcal O$-operators in this new setting.

The results obtained in this work enrich the theory of Hom-type superalgebras by providing new algebraic structures, new operator correspondences, and new construction techniques. They also reveal deep connections among Hom-Novikov superalgebras, Hom-anti-pre-Lie superalgebras, admissible operator pairs, and super-commutative cocycle structures, thereby opening new perspectives for further developments in Hom-superalgebra theory and its applications to mathematical physics and deformation theory.\\

Throughout this paper, unless otherwise specified, all vector spaces are assumed to be
finite-dimensional over a field $\mathbb{K}$ of characteristic $0$.

\section{Some basics on Hom-anti-pre-Lie superalgebras}

\begin{df}(\cite{Faouzi-Abdenacer})\label{Ex-H-Lie-sup}
A \textbf{Hom-Lie superalgebra} is a triplet $(\mathcal{A}, [\cdot,\cdot],\alpha)$ consisting of a $\mathbb{Z}_2$-graded vector space $\mathcal{A}$, an even bilinear map $[\cdot,\cdot] : \mathcal{A}\otimes \mathcal{A} \longrightarrow \mathcal{A},~~( \;[\mathcal{A}_i,\mathcal{A}_j]\subseteq \mathcal{A}_{i+j},
~~\forall~~i,j\in \mathbb{Z}_2\;)$ and an even linear map  $\alpha:\mathcal{A}\rightarrow\mathcal{A}$  satisfying:
\begin{eqnarray}
 \label{H-skwesym}
 &&[x,y] = -(-1)^{|x||y|}[y,x]~~ \text{( super-skew-symmetry)},\\
\label{H-sJ}
&&  [\alpha(x),[y,z]]=[[x,y],\alpha(z)]+(-1)^{|x||y|}[\alpha(y),[x,z]]~~ \text{(Hom super-Jacobi identity),}
\end{eqnarray}
$\forall\ x,y,z \in \mathcal{H}(\mathcal{A})$.
\end{df}
\begin{rem}
We recover the classical Lie superalgebra  when $\alpha=id$.
\end{rem}
\begin{exa}
Let $\mathfrak{g} = \mathfrak{g}_0 \oplus \mathfrak{g}_1$ be a 3-dimensional $\mathbb{Z}_2$-graded vector space over a field $\mathbb{K}$, with basis $\{e_1, e_2, e_3\}$ defined by: $$\mathfrak{g}_0 = \text{span}\{e_1\} \quad \text{and} \quad \mathfrak{g}_1 = \text{span}\{e_2, e_3\}.$$

Define the bilinear operation $[\cdot, \cdot] : \mathfrak{g} \times \mathfrak{g} \to \mathfrak{g}$ on basis elements by:
$$
[e_1, e_2] = -[e_2, e_1] = e_3,
$$
with all other basis brackets vanishing identically.\\ 
Define an even linear map $\alpha : \mathfrak{g} \to \mathfrak{g}$ by:
$$
\alpha(e_1) = e_1 , \quad \alpha(e_2) = e_2 + e_3, \quad \alpha(e_3) = e_3.
$$
Then, $(\mathfrak{g}, [\cdot, \cdot], \alpha)$ is a $3$-dimensional Hom-Lie superalgebra.    
\end{exa}

\begin{df}(\cite{wi-oth})
A \textbf{Hom-anti-pre-Lie superalgebra} (also called \textbf{Hom-anti-left-symmetric superalgebra}) is a triplet $(\mathcal A,\circ,\alpha)$ consisting of a $\mathbb Z_2$-graded vector space $\mathcal{A}$, an even bilinear map $\circ:\mathcal A\times\mathcal A\to\mathcal A$ and an even linear map $\alpha:\mathcal A\to\mathcal A$ such that, the following conditions hold   \begin{eqnarray}
 &&\alpha(x)\circ(y\circ z)-(-1)^{|x||y|} \alpha(y)\circ(x\circ z)=[y,x]_\circ \circ \alpha(z),\label{cond-Hom-ant-pre-Lie1}  \\
 &&(-1)^{|x||z|}[x,y]_\circ\circ \alpha(z)+(-1)^{|x||y|}[y,z]_\circ\circ \alpha(x)+(-1)^{|y||z|}[z,x]_\circ\circ \alpha(y)=0\label{cond-Hom-ant-pre-Lie2},
\end{eqnarray} 
for any $x,y,z\in\mathcal{H}(\mathcal{A})$, where \begin{equation}\label{comp-Lie-struc}
[x,y]_\circ=x\circ y-(-1)^{|x||y|}y\circ x.
\end{equation}
A Hom-anti-pre-Lie superalgebra $(\mathcal A,\circ,\mu)$ is called \textbf{regular} if $\alpha$ is bijective, and called \textbf{multiplicative} if $\alpha$ is an algebra endomorphism, i.e., $\alpha(x\circ y)=\alpha(x)\circ\alpha(y)$ for any $x,y\in\mathcal A$.

\end{df}
\begin{rem}\
 \begin{enumerate}
\item If $\alpha=id$, we recover \textbf{anti-pre-Lie superalgebras} structures.
\item Condition \eqref{cond-Hom-ant-pre-Lie1} is equivalent to
\begin{equation}\label{cond-equiv-Hom-anti-pre-Lie1}
\mathfrak{aas}_\circ^\alpha(x,y,z)=(-1)^{|x||y|}\mathfrak{aas}_\circ^\alpha(y,x,z),\;\forall x,y,z\in\mathcal{H}(\mathcal A). 
\end{equation}
 \end{enumerate}   
\end{rem}
\begin{exa}(\cite{wi-oth})
Let $\mathcal A=\mathcal A_{\overline 0}\oplus \mathcal A_{\overline 1}$ be a
$\mathbb Z_2$-graded vector space defined by
\[
\mathcal A_{\overline 0}=\langle e_1,e_2\rangle,
\qquad
\mathcal A_{\overline 1}=\langle e_3\rangle .
\]
Define an even bilinear map $
\circ:\mathcal A\otimes \mathcal A \longrightarrow \mathcal A
$
by
\[
e_1\circ e_2 = e_3,\qquad
e_2\circ e_1 = -\,e_3,
\]
and all other products are zero.
Let $\alpha:\mathcal A\to\mathcal A$ be the even linear map defined by
\[
\alpha(e_1)=e_1,\qquad
\alpha(e_2)=2e_2,\qquad
\alpha(e_3)=2e_3.
\]
Then  the triplet $(\mathcal A,\circ,\alpha)$
is a $3$-dimensional Hom-anti-pre-Lie superalgebra.
\end{exa}

Recall that $(\mathcal A,\circ,\alpha)$ is called \textbf{Hom-Lie-admissible superalgebra} where $\mathcal A$ is vector superspace with an even bilinear map $\mu:\mathcal A\times\mathcal A\to\mathcal A$ and an even linear map $\alpha:\mathcal A\to\mathcal A$, if the product $[\cdot,\cdot]:\mathcal A\times\mathcal A\to\mathcal A$ defined by Eq. \eqref{comp-Lie-struc} makes $(\mathcal A,[\cdot,\cdot],\alpha)$ a Hom-Lie superalgebra (see \cite{Faouzi-Abdenacer} for more details). We call $(\mathcal A,[\cdot,\cdot],\alpha)$ in this case the \textbf{sub-adjacent Hom-Lie superalgebra} of  $(\mathcal A,\circ,\alpha)$ denoted by $\mathfrak g(\mathcal A)$ and $(\mathcal A,\circ,\alpha)$ is called a \textbf{compatible} (Hom-Lie admissible) superalgebra structure on the Hom-Lie superalgebra $\mathfrak g(\mathcal A)$.
\begin{prop}\label{equiv-H-ant--pre-Lie-Hom-Lie admiss}
Let $\mathcal A$ be a vector superspace with an even bilinear map $\circ:\mathcal A\times\mathcal A\to\mathcal A$ and an even linear map $\alpha:\mathcal A\to\mathcal A$. Then the
following assertions are equivalent: 
\begin{enumerate}
\item $(\mathcal A,\circ,\alpha)$ is a Hom-anti-pre-Lie superalgebra.
\item For $(\mathcal A,\circ,\alpha)$, Eq. \eqref{cond-Hom-ant-pre-Lie1} and for any $x,y,z\in\mathcal H(\mathcal A)$, the following equation hold
\begin{equation}\label{cond-equiv1}
\circlearrowleft_{x,y,z}(-1)^{|x||z|}\alpha(x)\circ[y,z]_\circ=(-1)^{|x||z|}\alpha(x)\circ[y,z]_\circ+ (-1)^{|x||y|}\alpha(y)\circ[z,x]_\circ+(-1)^{|y||z|}\alpha(z)\circ[x,y]_\circ=0.
\end{equation}
\item $(\mathcal A,\circ,\alpha)$ is a Hom-Lie admissible superalgebra, such that $(\mathcal A,-\mathfrak L_\circ,\alpha)$ is a representation of the sub-adjacent Hom-Lie superalgebra $(\mathfrak g(\mathcal A),[\cdot,\cdot]_\circ,\alpha)$, where $\mathfrak L_\circ:\mathfrak g(\mathcal A)\to End(\mathcal A)$ is an even linear map defined by $\mathfrak L_\circ(x)(y)=x\circ y,\;\forall x,y\in\mathcal H(\mathcal A)$.
\end{enumerate}
\end{prop}
\begin{rem}
Recall that a Hom-pre-Lie superalgebra is a triple
$(A,\star,\alpha)$, where $\mathcal A=\mathcal A_{\bar 0}\oplus \mathcal A_{\bar 1}$ is a
$\mathbb{Z}_2$-graded vector space, $\alpha:\mathcal A\rightarrow \mathcal A$ is an
even linear map and $\star:\mathcal A\times \mathcal A\rightarrow \mathcal A$ is an even
bilinear map satisfying
\begin{equation}\label{HPL}
(x\star y)\star \alpha(z)-\alpha(x)\star (y\star z)
=(-1)^{|x||y|}
\Big((y\star x)\star \alpha(z)-\alpha(y)\star (x\star z)\Big),
\end{equation}
for all homogeneous elements $x,y,z\in \mathcal H(\mathcal A)$.

That is, a Hom-pre-Lie superalgebra $(\mathcal A,\star,\alpha)$ is a
Hom-Lie-admissible superalgebra such that
$(\mathcal A,\mathfrak L_{\star},\alpha)$ is a representation of the sub-adjacent
Hom-Lie superalgebra
$(\mathfrak g(\mathcal A),[\cdot,\cdot],\alpha)$, where
\[
[x,y]=x\star y-(-1)^{|x||y|}y\star x,
\qquad \forall x,y\in \mathcal H(\mathcal A).
\]

Therefore the notion of Hom-anti-pre-Lie superalgebras is justified,
since a Hom-anti-pre-Lie superalgebra $(\mathcal A,\circ,\alpha)$ is a
Hom-Lie-admissible superalgebra such that
$(\mathcal A,-\mathfrak L_{\circ},\alpha)$ is a representation of the sub-adjacent
Hom-Lie superalgebra
$(\mathfrak g(A),[\cdot,\cdot],\alpha)$.

Note that for a Hom-pre-Lie superalgebra $(\mathcal A,\star,\alpha)$,
the Hom-pre-Lie super-identity \eqref{HPL} is sufficient to ensure
that the supercommutator bracket defines a Hom-Lie superalgebra
structure on $\mathfrak g(\mathcal A)$. However, for a Hom-anti-pre-Lie
superalgebra $(\mathcal A,\circ,\alpha)$, the representation condition
corresponding to
\[
\alpha(x)\circ (y\circ z)
-(-1)^{|x||y|}\alpha(y)\circ (x\circ z)
=[y,x]\circ \alpha(z)
\]
alone is not sufficient to guarantee the Hom-super Jacobi identity.
Hence an additional compatibility condition, namely the Hom-super
analogue of Eq.~\eqref{cond-Hom-ant-pre-Lie2} (or equivalently Eq.~\eqref{cond-equiv1}), is required.    
\end{rem}
\begin{df}(\cite{wi-oth})\label{def-ant-O-op-H-Liesup}
Let $(V,\rho,\beta)$ be a representation of a Hom-Lie superalgebra $(\mathcal A,[\cdot,\cdot],\alpha)$. An even linear map $T:V\to\mathcal A$ is called an \textbf{anti-super-$\mathcal{O}$-operator} on $(\mathcal A,[\cdot,\cdot],\alpha)$ associated to $(V,\rho,\beta)$ if it satisfies
\begin{eqnarray}
T\circ\beta&=&\alpha\circ T,\label{cond-ant-O-oper-Hom-Lie1} \\
\lbrack T(u),T(v)\rbrack&=&T\big((-1)^{|u||v|}\rho(T(v))u-\rho(T(u))v\big),\;\forall u,v\in\mathcal H(V).\label{cond-ant-O-oper-Hom-Lie2}
\end{eqnarray}
an anti-super-$\mathcal O$-operator $T$ is called \textbf{strong} if it satisfies:
\begin{equation}\label{cond-strong-O-Oper}
\rho([T(u),T(v)])\beta(w)+(-1)^{|u|(|v|+|w|)}\rho([T(v),T(w)])\beta(u)+(-1)^{|w|(|u|+|v|)}\rho([T(w),T(u)])\beta(v)=0,  
\end{equation}
for all $ u,v,w\in\mathcal H(V)$.\\
In particular, an anti-super-$\mathcal O$-operator $\mathcal R$ of $(\mathcal A,[\cdot,\cdot],\alpha)$ associated to the adjoint representation $(\mathcal A,ad,\alpha)$ is called an \textbf{anti-Rota-Baxter operator (of weight zero)}, that is, $\mathcal R:\mathcal A\to\mathcal A$ is an even linear map commuting with $\alpha$ satisfying:
\begin{equation}\label{cond-Rota-Baxter-Oper}
 [\mathcal R(y),\mathcal R(x)]= (-1)^{|x||y|}\mathcal R([\mathcal R(x),y]+[x,\mathcal R(y)]),\;\forall x,y\in\mathcal A.  
\end{equation}
An anti-Rota-Baxter operator $\mathcal R$ is called \textbf{strong}, if it satisfies
\begin{equation}\label{cond-strong-Rota-Baxter-Oper}
[[\mathcal R(x),\mathcal R(y)],\alpha(z)]+(-1)^{|x|(|y|+|z|)} [[\mathcal R(y),\mathcal R(z)],\alpha(x)]+(-1)^{|z|(|x|+|y|)}[[\mathcal R(z),\mathcal R(x)],\alpha(y)]=0,   
\end{equation}
for all $x,y,z\in\mathcal H(\mathcal A)$.
\end{df}
\begin{df}(\cite{wi-oth})\label{def-ant-sup-Der-H-Liesup}
Let $(V,\rho,\beta)$ be a representation of a Hom-Lie superalgebra $(\mathcal A,[\cdot,\cdot],\alpha)$. A linear map $\mathfrak D:\mathcal A\to V$ is called a \textbf{super anti-derivation} ( or a \textbf{super anti-$1$-cocycle}) if $\mathfrak D$ satisfies
\begin{eqnarray}
\mathfrak D\circ\alpha&=&\beta\circ\mathfrak D,\label{anti-der-Hom-Lie1}\\
\mathfrak D([x,y])&=&(-1)^{|y|(|\mathfrak D|+|x|)}\rho(y)(\mathfrak D(x))-\rho(x)(\mathfrak D(y)),\;\forall x,y\in\mathcal A.\label{anti-der-Hom-Lie2}    
\end{eqnarray}
\end{df}
\begin{prop}(\cite{wi-oth})\label{inver-ant-O-op-iff-anti-der}
Let $(V,\rho,\beta)$ be a representation of a Hom-Lie superalgebra $(\mathcal A,[\cdot,\cdot],\alpha)$ and $T:V\to\mathcal A$ an even invertible linear map. Then $T$ is an anti-super-$\mathcal O$-operator on $\mathcal A$ associated to $(V,\rho,\beta)$ if and only if $T^{-1}$ is an even anti-derivation on $\mathcal A$.
\end{prop}

\begin{thm}(\cite{wi-oth})\label{Hom-ant-pre-Lie-by-ant-O-oper}
 Let $T$ be an anti-super-$\mathcal O$-operator on a Hom-Lie superalgebra $(\mathcal A,[\cdot,\cdot],\alpha)$ with respect to a representation $(V,\rho,\beta)$. Then, $(V,\circ_V,\beta)$ where, $\circ_V:V\times V\to V$ is defined by
 \begin{equation}\label{ant-O-oper-to-anti-pre-Lie}
 u\circ_V v=-\rho(T(u))v,\;\forall u,v\in V,    
 \end{equation}
 satisfies Eq. \eqref{cond-Hom-ant-pre-Lie1}. Moreover, $(V,\circ_V,\beta)$ is a Hom-Lie-admissible such that $(V,\circ_V,\beta)$ is a Hom-anti-pre-Lie superalgebra if and only if $T$ is strong. So, $T$ is a homomorphism of Hom-Lie superalgebras from the sub-adjacent Hom-Lie superalgebra $(\mathcal A(V),[\cdot,\cdot]_{\circ_V},\beta)$ to $(\mathcal A,[\cdot,\cdot],\alpha)$. Therefore, there is an induced Hom-anti-pre-Lie
superalgebra structure on $T(V)=\{T(u);\;u\in V\}\subset\mathcal A$ given by
\begin{equation}\label{ind-Hom-anti-T(V)}
T(u)\circ_\mathcal AT(v)=T(u\circ_Vv),\forall u,v\in V,
\end{equation}
and $T$ is a homomorphism of Hom-anti-pre-Lie superalgebras.
\end{thm}

\begin{cor}(\cite{wi-oth})
Let $\mathcal R:\mathcal A\to\mathcal A$ be a strong anti-Rota-Baxter operator of a Hom-Lie superalgebra $(\mathcal A,[\cdot,\cdot],\alpha)$. Then, the product 
\begin{equation}\label{ant-R-B-to-anti-pre-Lie}
x\circ y=-[\mathcal R(x),y],\;\forall x,y\in\mathcal A,    
\end{equation}
defines on $\mathcal A$ a Hom-anti-pre-Lie superalgebra structure. Conversely, if $\mathcal R:\mathcal A\to\mathcal A$ is a linear transformation of a Hom-Lie superalgebra $(\mathcal A,[\cdot,\cdot],\alpha)$ such that, the product defined by Eq. \eqref{ant-R-B-to-anti-pre-Lie} defines on $\mathcal A$ a Hom-anti-pre-Lie superalgebra, then $\mathcal R$ satisfies Eq. \eqref{cond-strong-Rota-Baxter-Oper} and the following equation:
\begin{equation}\label{Cor-R-B-Hom-anti-pre-Lie}
[[\mathcal R(x),\mathcal R(y)]+\mathcal R([\mathcal R(x),y]+[x,\mathcal R(y)]),\alpha(z)]=0,\forall x,y,z\in\mathcal A.
\end{equation}
\end{cor}
\begin{prop}\label{Prop-inv-O-oper-autom-strong}
An invertible anti-super-$\mathcal O$-operator on a Hom-Lie superalgebra is autmatically strong.    
\end{prop}
\begin{proof}
  Let $T:\mathcal A\to V$ be an anti-super-$\mathcal O$-operator on a Hom-Lie superalgebra $(\mathcal A,[\cdot,\cdot]_\mathcal A,\alpha)$ associated to a representation $(V,\rho,\beta)$. Define a bilinear operation $\circ :V\times V\to V$ by Eq. \eqref{ant-O-oper-to-anti-pre-Lie}. Then we have:
  $$[u,v]_V=u\circ v-(-1)^{|u||v|}v\circ u=(-1)^{|u||v|}\rho(T(v))u-\rho(T(u))v=T^{-1}([T(u),T(v)]_V),\;\forall u,v\in\mathcal H(V).$$
  Therefore for all $u,v,w\in\mathcal H(V)$, we have:
  $$[[u,v]_V,\beta(w)]_V=[T^{-1}([T(u),T(v)]_\mathcal A),\beta(w)]_V=T^{-1}([[T(u),T(v)]_\mathcal A,\alpha(T(w))]_\mathcal A).$$
  Thus, $(V,[\cdot,\cdot],\beta)$ is a Hom-Lie superalgebra, and hence $(V,\circ,\beta)$ is a Hom-Lie-admissible superalgebra. Then by Theorem \ref{Hom-ant-pre-Lie-by-ant-O-oper}, $T$ is strong.
\end{proof}
\begin{cor}\label{cor-comp-h-anti-pre-Lie-inv-O-op}
 Let $(\mathcal A, [\cdot,\cdot],\alpha)$ be a Hom-Lie superalgebra. Then there is a compatible Hom-anti-pre-Lie superalgebra
structure on $\mathcal A$ if and only if there exists an invertible anti-super-$\mathcal O$-operator on $(\mathcal A, [\cdot,\cdot],\alpha)$.   
\end{cor}

\begin{proof}
  Suppose that $(\mathcal A,\circ,\alpha)$ is a compatible
Hom-anti-pre-Lie superalgebra structure on the
Hom-Lie superalgebra $(\mathcal A,[\cdot,\cdot],\alpha)$.
Then
\[
[x,y]
=
x\circ y
-
(-1)^{|x||y|}y\circ x,
\qquad
\forall x,y\in \mathcal H(\mathcal A).
\]
Moreover,
\[
[x,y]
=
-(-\mathfrak L_\circ(y))x
+
(-1)^{|x||y|}(-\mathfrak L_\circ(x))y,
\qquad
\forall x,y\in \mathcal H(\mathcal A).
\]
Thus
\[
T=\mathrm{id}_{\mathcal A}
\]
is an invertible Hom-anti-$\mathcal O$-operator of
$(\mathcal A,[\cdot,\cdot],\alpha)$ associated to the
representation $(\mathcal A,-\mathfrak L_\circ,\alpha)$.

Conversely, suppose that
\[
T:V\longrightarrow \mathcal A
\]
is an invertible Hom-anti-$\mathcal O$-operator of
$(\mathcal A,[\cdot,\cdot],\alpha)$ associated to the
representation $(V,\rho,\beta)$.
Then by Proposition \ref{Prop-inv-O-oper-autom-strong},
$T$ is strong and by 
Theorem \ref{Hom-ant-pre-Lie-by-ant-O-oper}, there exists a Hom-anti-pre-Lie
superalgebra structure on $V$ given by
\[
u\circ v
=
-\rho(T(u))v,
\qquad
\forall u,v\in \mathcal H(V).
\]
Moreover, since $T$ is invertible, there is an induced
Hom-anti-pre-Lie superalgebra structure on the
underlying superspace of $\mathcal A$ defined by
\[
x\circ_{\mathcal A} y
=
T\bigl(T^{-1}(x)\circ T^{-1}(y)\bigr),
\qquad
\forall x,y\in \mathcal H(\mathcal A).
\]
Explicitly, for any homogeneous elements
$x,y\in \mathcal H(\mathcal A)$, there exist
$u,v\in \mathcal H(V)$ such that
$x=T(u)$ and $y=T(v)$. Hence
\begin{align*}
x\circ_{\mathcal A} y
&=
T(u)\circ_{\mathcal A}T(v)\\
&=
T(u\circ v)\\
&=
-T\bigl(\rho(T(u))v\bigr)\\
&=
-T\bigl(\rho(x)T^{-1}(y)\bigr).
\end{align*}
Furthermore, by the definition of an anti-super-$\mathcal O$-operator,
we have
\begin{align*}
[x,y]
&=
[T(u),T(v)]\\
&=
T\Bigl(
\rho(T(v))u
-
(-1)^{|u||v|}
\rho(T(u))v
\Bigr)\\
&=
T\Bigl(
\rho(y)T^{-1}(x)
-
(-1)^{|x||y|}
\rho(x)T^{-1}(y)
\Bigr)\\
&=
x\circ_{\mathcal A} y
-
(-1)^{|x||y|}
y\circ_{\mathcal A} x.
\end{align*}
Therefore,
$(\mathcal A,\circ_{\mathcal A},\alpha)$
is a Hom-anti-pre-Lie superalgebra whose
sub-adjacent Hom-Lie superalgebra is
$(\mathcal A,[\cdot,\cdot],\alpha)$.
  
\end{proof}
\begin{df}(\cite{wi-oth1})
Let $(\mathcal A,[\cdot,\cdot],\alpha)$ be a Hom-Lie superalgebra, and $\mathfrak B$ an even
super-commutative bilinear form on $\mathcal A$, such that for any $x,y,z\in \mathcal H(\mathcal A)$, the following condition hold:.
\begin{equation}\label{sup-comm-2-cocyc}
\mathfrak B(\alpha(x),[y,z])
+(-1)^{|x|(|y|+|z|)}\mathfrak B(\alpha(y),[z,x])
+(-1)^{|z|(|x|+|y|)}\mathfrak B(\alpha(z),[x,y])=0.
\end{equation}
Then the bilinear form $\mathfrak B$ is called a
\textbf{super-commutative $2$-cocycle} on the Hom-Lie superalgebra $\mathcal A$.    
\end{df}
\begin{thm}(\cite{wi-oth1})\label{ant-h-pr-Lie-from-nondeg-bil}
  Let $\mathfrak{B}$ be a nondegenerate super-commutative $2$-cocycle on a multiplicative Hom-Lie superalgebra $(\mathcal A,[\cdot,\cdot],\alpha)$. Then there exists a compatible Hom-anti-pre-Lie superalgebra structure $\circ$ on $(\mathcal A,[\cdot,\cdot],\alpha)$ given by \begin{equation}\label{inv-Hom-anti-pre-Lie}\mathfrak{B}(x\circ y,\alpha(z))=(-1)^{|x||y|}\mathfrak{B}(\alpha(y),[x,z]), \;\;\forall x,y,z\in \mathcal H(\mathcal A). \end{equation}  
\end{thm}

\section{Hom-Novikov superalgebras and admissible Hom-Novikov superalgebras}

We introduce the notion of admissible Hom-Novikov superalgebras as a
subclass of Hom-anti-pre-Lie superalgebras. Such a
correspondence gives the construction of Hom-anti-pre-Lie superalgebras
from commutative associative superalgebras with derivations or
admissible pairs and leads to the introduction of the notions of Hom-anti-pre-Lie Poisson superalgebras and admissible Hom-Novikov-Poisson superalgebras.

\subsection{Correspondence between Hom-Novikov superalgebras and admissible Hom-Novikov superalgebras}

We introduce the notion of admissible Hom-Novikov superalgebras as a
subclass of Hom-anti-pre-Lie superalgebras, whose name comes from a
correspondence between them and Hom-Novikov superalgebras in terms of
$q$-algebras. The interpretation of admissible Hom-Novikov superalgebras in
terms of the corresponding anti-super-$\mathcal O$-operators and some
examples are given.

\begin{df}\label{defi:Hom-Novikov superalgebras}(\cite{Zhang-Hou-Bai})
A \textbf{Hom-Novikov superalgebra} is a Hom-pre-Lie superalgebra $(\mathcal A,\star,\alpha)$ such that
\begin{equation}\label{eq:defi:Hom-Novikov superalgebras1}
(x\star y)\star \alpha(z)=(-1)^{|y||z|}(x\star z)\star \alpha(y), \forall x,y,z\in\mathcal H(\mathcal A).
\end{equation}
\end{df}
\begin{exa}\label{ex-H-Nov-sup}
Let $\mathcal{A} = \mathcal{A}_{\bar{0}} \oplus \mathcal{A}_{\bar{1}}$ be a $3$-dimensional $\mathbb{Z}_2$-graded vector space, spanned by the basis $\{e_1, e_2, f\}$ with degrees $|e_1| = |e_2| = \bar{0}$ and $|f| = \bar{1}$.

Define the bilinear product $\star: \mathcal{A} \times \mathcal{A} \to \mathcal{A}$ on the basis elements by:
\begin{equation*}
e_1 \star e_1 = e_2, \quad f \star f = e_2,
\end{equation*}
with all other products between basis elements being zero.

Define the linear map $\alpha: \mathcal{A} \to \mathcal{A}$ by:
\begin{equation*}
\alpha(e_1) = e_1, \quad \alpha(e_2) = e_2, \quad \alpha(f) = -f.
\end{equation*}

Then $(\mathcal{A}, \star, \alpha)$ is a Hom-Novikov superalgebra.
\end{exa}

\begin{df}\label{defi:admissible Hom-Novikov superalgebras} An \textbf{admissible
Hom-Novikov superalgebra} is a triplet $(\mathcal A,\circ,\alpha)$ consisting of a $\mathbb Z_2$-graded vector space $\mathcal A$, an even bilinear map $\circ:\mathcal A\times \mathcal A
\rightarrow \mathcal A$ and an even linear map $\alpha:\mathcal A\to\mathcal A$ such that Eq. \eqref{cond-Hom-ant-pre-Lie1} and the following equation hold:
\begin{equation}\label{admiss-Hom-Nov-sup2}
 2\alpha(x)\circ[y,z]_\circ=(x\circ y)\circ \alpha(z)-(-1)^{|y||z|}(x\circ z)\circ \alpha(y), 
\end{equation} 
for any $x,y,z\in\mathcal{H}(\mathcal{A})$, where $[\cdot,\cdot]_\circ$ is defined by Eq. \eqref{comp-Lie-struc}.
\end{df}

\begin{prop}\label{H-admissible is H-anti-pre-Lie}
An admissible Hom-Novikov superalgebra is a Hom-anti-pre-Lie superalgebra.
\end{prop}
\begin{proof}
Let $(\mathcal A,\circ,\alpha)$ be an admissible Hom-Novikov superalgebra, and $x,y,z\in
\mathcal H(\mathcal A)$. By Eq.~\eqref{cond-Hom-ant-pre-Lie1} and Proposition \ref{equiv-H-ant--pre-Lie-Hom-Lie admiss},
we have
\begin{align*}
\alpha(x)\circ[y,z]+(-1)^{|x|(|y|+|z|)}\alpha(y)\circ[z,x]+(-1)^{|z|(|x|+|y|)}\alpha(z)\circ[x,y]&=(-1)^{|x||y|}[y,x]\circ \alpha(z)\\&+(-1)^{|x|(|y|+|z|)+|y||z|}[z,y]\circ \alpha(x)\\&+(-1)^{|y||z|}[x,z]\circ \alpha(y).
\end{align*}

On the other hand, by Eq.~(\ref{admiss-Hom-Nov-sup2}), we have
\begin{eqnarray*}
&&2\alpha(x)\circ[y,z]+2(-1)^{|x|(|y|+|z|)}\alpha(y)\circ[z,x]+2(-1)^{|z|(|x|+|y|)}\alpha(z)\circ[x,y]\\
&&=(x\circ y)\circ\alpha(z)-(-1)^{|y||z|}(x\circ z)\circ\alpha(y)+(-1)^{|x|(|y|+|z|)}(y\circ z)\circ\alpha(x)\\&&-(-1)^{|x||y|}(y\circ x)\circ\alpha(z)+(-1)^{|z|(|x|+|y|)}(z\circ x)\circ\alpha(y)-(-1)^{|x|(|y|+|z|)+|y||z|}(z\circ y)\circ\alpha(x)\\
&&=[x,y]\circ\alpha(z)+(-1)^{|x|(|y|+|z|)}[y,z]\circ\alpha(x)+(-1)^{|z|(|x|+|y|)}[z,x]\circ\alpha(y).
\end{eqnarray*}

Thus
\begin{align*}
&\alpha(x)\circ[y,z]+(-1)^{|x|(|y|+|z|)}\alpha(y)\circ[z,x]+(-1)^{|z|(|x|+|y|)}\alpha(z)\circ[x,y]\\&=[x,y]\circ\alpha(z)+(-1)^{|x|(|y|+|z|)}[y,z]\circ\alpha(x)+(-1)^{|z|(|x|+|y|)}[z,x]\circ\alpha(y)\\&=0.
\end{align*}
Hence $(A,\circ,\alpha)$ is a Hom-anti-pre-Lie superalgebra.
\end{proof}

We introduce the notion of ``admissible Hom-Novikov superalgebras" since
they ``correspond" to Hom-Novikov superalgebras in the following sense:
Hom-Novikov algebras and admissible Hom-Novikov superalgebras can be realized
as $q$-Hom-algebras  each other.

\begin{df}\label{defi:q-algebra}
Let $\mathcal A$ be a $\mathbb Z_2$-vector space with an even bilinear operation
$\star:\mathcal A\times \mathcal A \rightarrow \mathcal A$. Define a bilinear operation
$\circ:\mathcal A\times \mathcal A\rightarrow \mathcal A$ by
\begin{equation}\label{eq:defi:q-algebra}
x\circ y=x\star y+(-1)^{|x||y|}qy\star x, \forall x,y\in\mathcal H(\mathcal A),
\end{equation}
for some $q\in\mathbb{F}$. Then $(\mathcal A,\circ,\alpha)$ is called the
\textbf{$q$-Hom-superalgebra} of $(A,\star,\alpha)$.
\end{df}

\begin{prop}\label{pro:anti-pre-Lie algebras from Novikov algebras}
Let $(\mathcal A,\star,\alpha)$ be a Hom-pre-Lie superalgebra, and $(\mathcal A,\circ,\alpha)$ be the
$2$-superalgebra of $(\mathcal A,\star,\alpha)$, that is,
\begin{equation}\label{eq:pro:H-anti-pre-Lie superalgebras from H-Novikov superalgebras}
x\circ y=x\star y+(-1)^{|x||y|}2y\star x,\forall x,y\in\mathcal H(\mathcal A).
\end{equation}
Then $(\mathcal A,\circ,\alpha)$ is a Hom-anti-pre-Lie superalgebra if and only if
$(\mathcal A,\star,\alpha)$ is further a Hom-Novikov superalgebra. Moreover, in this case,
$(\mathcal A,\circ,\alpha)$ is an admissible Hom-Novikov superalgebra.
\end{prop}
\begin{proof}
Let the sub-adjacent Hom-Lie superalgebra of $(\mathcal A,\star,\alpha)$ be $(\frak
g(A),\lbrace-,-\rbrace,\alpha)$. Then by Eq.~(\ref{eq:pro:H-anti-pre-Lie superalgebras from H-Novikov superalgebras}), we have
\begin{equation*}\label{eq:pro:anti-pre-Lie algebras from Novikov algebras2}
[x,y]=x\circ y-(-1)^{|x||y|}y\circ x=(-1)^{|x||y|}y\star x-x\star y=(-1)^{|x||y|}\lbrace y,x\rbrace,
\forall x,y\in\mathcal H(\mathcal A).
\end{equation*}
Thus $(\mathcal A,\circ,\alpha)$ is a Hom-Lie-admissible superalgebra.\\
Let $x,y,z\in\mathcal H(\mathcal A)$. Then we have

\begin{eqnarray*}
&&\alpha(x)\circ(y\circ z)
-(-1)^{|x||y|}
\alpha(y)\circ(x\circ z)
+[x,y]\circ \alpha(z)\\
&&\overset{(\ref{eq:pro:H-anti-pre-Lie superalgebras from H-Novikov superalgebras})}{=}\
\alpha(x)\star(y\star z)
+2(-1)^{|x|(|y|+|z|)}(y\star z)\star \alpha(x)
+2(-1)^{|y||z|}
\alpha(x)\star(z\star y)
\\
&&\hspace{0.4cm}+4(-1)^{|x|(|y|+|z|)+|y||z|}
(z\star y)\star \alpha(x)
-(-1)^{|x||y|}
\alpha(y)\star(x\star z)
-2(-1)^{|y||z|}
(x\star z)\star \alpha(y)\\
&&\hspace{0.4cm}
-2(-1)^{|x|(|y|+|z|)}
\alpha(y)\star(z\star x)
-4(-1)^{|z|(|x|+|y|)}
(z\star x)\star \alpha(y)
+(-1)^{|x||y|}
(y\star x)\star \alpha(z)
\\
&&\hspace{0.4cm}+2(-1)^{|x|(|y|+|z|)+|y||z|}\alpha(z)\star(y\star x)
-(x\star y)\star \alpha(z)
-2(-1)^{|z|(|x|+|y|)}
\alpha(z)\star(x\star y)\\
&&\overset{(\ref{HPL})}{=}\
2(-1)^{|x|(|y|+|z|)}
(y\star z)\star \alpha(x)
-2(-1)^{|x|(|y|+|z|)}
\alpha(y)\star(z\star x)
+(-1)^{|x|(|y|+|z|)+|y||z|}2\alpha(z)\star(y\star x)
\\
&&\hspace{0.4cm}-2(-1)^{|y||z|}
(x\star z)\star \alpha(y)
+2(-1)^{|y||z|}
\alpha(x)\star(z\star y)
-2(-1)^{|z|(|x|+|y|)}
\alpha(z)\star(x\star y)\\
&&\hspace{0.4cm}
+4(-1)^{|x|(|y|+|z|)+|y||z|}
(z\star y)\star \alpha(x)
-4(-1)^{|z|(|x|+|y|)}
(z\star x)\star \alpha(y)\\
&&\overset{(\ref{HPL})}{=}\
6(-1)^{|x|(|y|+|z|)+|y||z|}
(z\star y)\star \alpha(x)
-6(-1)^{|z|(|x|+|y|)}
(z\star x)\star \alpha(y).
\end{eqnarray*}
Then $(\mathcal A,\circ,\alpha)$ is a Hom-anti-pre-Lie superalgebra if and only if
$(\mathcal A,\star,\alpha)$ is a Hom-Novikov superalgebra. Moreover, if  $(\mathcal A,\star,\alpha)$ is a
Hom-Novikov superalgebra, then we have
\begin{eqnarray*}
&&2\alpha(x)\circ[y,z]-(x\circ y)\circ \alpha(z)+(-1)^{|y||z|}(x\circ z)\circ \alpha(y)\\
&&=2(-1)^{|y||z|}\alpha(x)\circ(z\star y)-2\alpha(x)\circ(y\star z)-(x\circ y)\circ \alpha(z)+(-1)^{|y||z|}(x\circ z)\circ \alpha(y)\\
&&\overset{(\ref{eq:pro:H-anti-pre-Lie superalgebras from H-Novikov superalgebras})}{=}2\alpha(x)(-1)^{|y||z|}\star(z\star y)+4(-1)^{|x|(|y|+|z|)+|y||z|}(z\star y)\star \alpha(x)-2\alpha(x)\star(y\star
z)\\&&\hspace{0.4cm}-(-1)^{|x|(|y|+|z|)}4(y\star z)\star \alpha(x)
-(x\star y)\star \alpha(z)-2(-1)^{|z|(|x|+|y|)}\alpha(z)\star(x\star y)-(-1)^{|x||y|}2(y\star x)\star \alpha(z)\\
&&\hspace{0.4cm}-4(-1)^{|x|(|y|+|z|)+|y||z|}\alpha(z)\star(y\star x)+(x\star z)\star \alpha(y)+2(-1)^{|x||y|}\alpha(y)\star(x\star z)\\
&&\hspace{0.4cm}+2(-1)^{|z|(|x|+|y|)}(z\star x)\star \alpha(y)+4(-1)^{|x|(|y|+|z|)}\alpha(y)\star(z\star x)\\
&&\overset{(\ref{HPL}),(\ref{eq:defi:Hom-Novikov superalgebras1})}{=}2(-1)^{|y||z|}(x\star z)\star \alpha(y)-2(x\star y)\star \alpha(z)\\
&&=0.
\end{eqnarray*}
Thus $(\mathcal A,\circ,\alpha)$ is an admissible Hom-Novikov superalgebra.
\end{proof}


By Eq.~(\ref{eq:pro:H-anti-pre-Lie superalgebras from H-Novikov superalgebras}),
there is an equivalent expression:
\begin{equation}\label{eq:pro:admissible Novikov algebras and Novikov algebras1}
x\star y=-\dfrac{1}{3}x\circ y+\dfrac{2}{3}(-1)^{|x||y|}y\circ x, \forall
x,y\in\mathcal H(\mathcal A).
\end{equation}
In terms of $q$-Hom-superalgebras, we adjust the above equation to be
\begin{equation}\label{eq:pro:admissible Novikov algebras and Novikov algebras2}
x\star y=x\circ y-2(-1)^{|x||y|}y\circ x, \forall x,y\in\mathcal H(\mathcal A).
\end{equation}

\begin{prop}\label{pro:admissible Novikov algebras and Novikov algebras}

Let $(\mathcal A,\circ,\alpha)$ be a Hom-anti-pre-Lie superalgebra, and let
$(\mathcal A,\star,\alpha)$ be the $(-2)$-superalgebra of $(\mathcal A,\circ,\alpha)$ given by Eq.~(\ref{eq:pro:admissible
Novikov algebras and Novikov algebras2}).
Then $(\mathcal A,\star,\alpha)$ is a Hom-pre-Lie superalgebra if and only if
$(\mathcal A,\circ,\alpha)$ is an admissible Hom-Novikov superalgebra. Moreover, in this case, $(\mathcal A,\star,\alpha)$ is a Hom-Novikov superalgebra.

\end{prop}
\begin{proof}
Let $x,y,z\in \mathcal{H}(\mathcal A)$ be homogeneous elements. Then we have
\begin{eqnarray*}
&&(x\star y)\star \alpha(z)-\alpha(x)\star (y\star z)\\
&&\hspace{0.4cm}-(-1)^{|x||y|}\Big((y\star x)\star \alpha(z)-\alpha(y)\star (x\star z)\Big)\\
&&\overset{(\ref{eq:pro:admissible Novikov algebras and Novikov algebras2})}{=}
(x\circ y)\circ \alpha(z)
-2(-1)^{|z|(|x|+|y|)}\alpha(z)\circ (x\circ y)\\
&&\hspace{0.4cm}
-2(-1)^{|x||y|}(y\circ x)\circ \alpha(z)
+4(-1)^{|x||y|+|z|(|x|+|y|)}\alpha(z)\circ (y\circ x)\\
&&\hspace{0.4cm}
-\alpha(x)\circ (y\circ z)
+2(-1)^{|x|(|y|+|z|)}(y\circ z)\circ \alpha(x)\\
&&\hspace{0.4cm}
+2\alpha(x)\circ (z\circ y)
-4(-1)^{|y||z|}(z\circ y)\circ \alpha(x)\\
&&\hspace{0.4cm}
-(-1)^{|x||y|}(y\circ x)\circ \alpha(z)
+2(-1)^{|x||y|+|z|(|x|+|y|)}\alpha(z)\circ (y\circ x)\\
&&\hspace{0.4cm}
+2(-1)^{|x||y|}(x\circ y)\circ \alpha(z)
-4(-1)^{|x||y|+|z|(|x|+|y|)}\alpha(z)\circ (x\circ y)\\
&&\hspace{0.4cm}
+(-1)^{|x||y|}\alpha(y)\circ (x\circ z)
-2(-1)^{|x||y|+|y|(|x|+|z|)}(x\circ z)\circ \alpha(y)\\
&&\hspace{0.4cm}
-2(-1)^{|x||y|}\alpha(y)\circ (z\circ x)
+4(-1)^{|x||y|+|z||x|}(z\circ x)\circ \alpha(y)\\
&&\overset{(\ref{cond-Hom-ant-pre-Lie1})}{=}
4[x,y]\circ \alpha(z)
-6(-1)^{|z|(|x|+|y|)}\alpha(z)\circ [x,y]\\
&&\hspace{0.4cm}
+2(-1)^{|x|(|y|+|z|)}[y,z]\circ \alpha(x)
+2(-1)^{|z|(|x|+|y|)}[z,x]\circ \alpha(y)\\
&&\hspace{0.4cm}
-2(-1)^{|y||z|}(z\circ y)\circ \alpha(x)
+2(-1)^{|z||x|}(z\circ x)\circ \alpha(y)\\
&&\hspace{0.4cm}
+2\alpha(x)\circ (z\circ y)
-2(-1)^{|x||y|}\alpha(y)\circ (z\circ x)\\
&&\overset{(\ref{cond-equiv1})}{=}
4[x,y]\circ \alpha(z)
-6(-1)^{|z|(|x|+|y|)}\alpha(z)\circ [x,y]\\
&&\hspace{0.4cm}
+2(-1)^{|x|(|y|+|z|)}[y,z]\circ \alpha(x)
+2(-1)^{|z|(|x|+|y|)}[z,x]\circ \alpha(y)\\
&&\hspace{0.4cm}
-2(-1)^{|y||z|}(z\circ y)\circ \alpha(x)
+2(-1)^{|z||x|}(z\circ x)\circ \alpha(y)\\
&&\hspace{0.4cm}
+2(-1)^{|x||y|}[y,x]\circ \alpha(z)
-2(-1)^{|z|(|x|+|y|)+|x||y|}\alpha(z)\circ [y,x]\\
&&\overset{(\ref{cond-Hom-ant-pre-Lie2})}{=}
-4(-1)^{|z|(|x|+|y|)}\alpha(z)\circ [x,y]\\
&&\hspace{0.4cm}
+2(-1)^{|z||x|}(z\circ x)\circ \alpha(y)
-2(-1)^{|y||z|}(z\circ y)\circ \alpha(x).
\end{eqnarray*}
Thus $(\mathcal A,\star,\alpha)$ is a Hom-pre-Lie superalgebra if and only if
$(\mathcal A,\circ,\alpha)$ is an admissible Hom-Novikov superalgebra. Moreover, if $(\mathcal A,\circ,\alpha)$ is an admissible Hom-Novikov
superalgebra, then for any homogeneous elements
$x,y,z\in \mathcal{H}(\mathcal A)$, we have
\begin{eqnarray*}
&&(x\star y)\star \alpha(z)
-(-1)^{|y||z|}(x\star z)\star \alpha(y)\\
&&\overset{(\ref{eq:pro:admissible Novikov algebras and Novikov algebras2})}{=}
(x\circ y)\circ \alpha(z)
-2(-1)^{|z|(|x|+|y|)}\alpha(z)\circ (x\circ y)\\
&&\hspace{0.4cm}
-2(-1)^{|x||y|}(y\circ x)\circ \alpha(z)
+4(-1)^{|x||y|+|z|(|x|+|y|)}\alpha(z)\circ (y\circ x)\\
&&\hspace{0.4cm}
-(-1)^{|y||z|}(x\circ z)\circ \alpha(y)
+2(-1)^{|x||z|+|y|(|x|+|z|)}\alpha(y)\circ (x\circ z)\\
&&\hspace{0.4cm}
+2(-1)^{|y||z|}(z\circ x)\circ \alpha(y)
-4(-1)^{|y||z|+|x||z|}\alpha(y)\circ (z\circ x)\\
&&\overset{(\ref{comp-Lie-struc})}{=}
[x,y]\circ \alpha(z)
-(-1)^{|x||y|}(y\circ x)\circ \alpha(z)\\
&&\hspace{0.4cm}
-2(-1)^{|z|(|x|+|y|)}\alpha(z)\circ [x,y]
-(-1)^{|y||z|}[x,z]\circ \alpha(y)\\
&&\hspace{0.4cm}
+(-1)^{|y||z|+|x||z|}(z\circ x)\circ \alpha(y)
+2(-1)^{|x|(|y|+|z|)}\alpha(y)\circ [x,z]\\
&&\hspace{0.4cm}
-2(-1)^{|x||z|+|y|(|x|+|z|)}\alpha(y)\circ (z\circ x)\\
&&\overset{(\ref{cond-Hom-ant-pre-Lie1})}{=}
[x,y]\circ \alpha(z)
+(-1)^{|z|(|x|+|y|)}[z,x]\circ \alpha(y)\\
&&\hspace{0.4cm}
+2(-1)^{|x|(|y|+|z|)}[y,z]\circ \alpha(x)
+2(-1)^{|z|(|x|+|y|)}\alpha(z)\circ [y,x]\\
&&\hspace{0.4cm}
+2(-1)^{|x|(|y|+|z|)}\alpha(y)\circ [x,z]\\
&&\hspace{0.4cm}
-(-1)^{|x||y|}(y\circ x)\circ \alpha(z)
+(-1)^{|y||z|+|x||z|}(z\circ x)\circ \alpha(y)\\
&&\overset{(\ref{cond-Hom-ant-pre-Lie2}),
(\ref{cond-equiv1})}{=}
(-1)^{|y||z|}[x,z]\circ \alpha(y)
+(-1)^{|x||y|}[y,x]\circ \alpha(z)\\
&&\hspace{0.4cm}
+2\alpha(x)\circ [y,z]
-(-1)^{|x||y|}(y\circ x)\circ \alpha(z)\\
&&\hspace{0.4cm}
+(-1)^{|y||z|+|x||z|}(z\circ x)\circ \alpha(y)\\
&&\overset{(\ref{comp-Lie-struc})}{=}
2\alpha(x)\circ [y,z]
+(-1)^{|y||z|}(x\circ z)\circ \alpha(y)
-(x\circ y)\circ \alpha(z)\\
&&\overset{(\ref{admiss-Hom-Nov-sup2})}{=}0.
\end{eqnarray*}
Thus $(\mathcal A,\star,\alpha)$ is a Hom-Novikov superalgebra.

\end{proof}

\begin{rem}
    If we only consider the categories of Hom-Novikov superalgebras and
admissible Hom-Novikov superalgebras, then the fact that
$(\mathcal A,\star,\alpha)$ is a Hom-Novikov superalgebra if and only if its
$2$-Hom-superalgebra $(\mathcal A,\circ,\alpha)$ is an admissible
Hom-Novikov superalgebra implies the ``converse'' side that
$(\mathcal A,\circ,\alpha)$ is an admissible Hom-Novikov superalgebra if and
only if its $(-2)$-Hom-superalgebra $(\mathcal A,\star,\alpha)$ is a
Hom-Novikov superalgebra. However, since the above correspondence is
placed into a broader framework involving Hom-pre-Lie
superalgebras and Hom-anti-pre-Lie superalgebras, we would like to
point out that Proposition~\ref{pro:admissible Novikov algebras and Novikov
algebras} cannot be obtained directly from
Proposition~\ref{pro:anti-pre-Lie algebras from Novikov algebras}.
\end{rem}

Propositions \ref{pro:anti-pre-Lie algebras from Novikov algebras}
and \ref{pro:admissible Novikov algebras and Novikov algebras}
yield the following relationship.

$$\mbox{\{Hom-pre-Lie\} $\hookleftarrow$ \{Hom-Novikov\}
$\overset{2-Hom-superalgebra}{\underset{(-2)-Hom-superalgebra}{\rightleftharpoons}}$
\{admissible Hom-Novikov\} $\hookrightarrow$ \{Hom-anti-pre-Lie\}.}$$

\begin{exa}
Any commutative Hom-associative superalgebra is both a Hom-Novikov superalgebra and
an admissible Hom-Novikov superalgebra.
\end{exa}

The correspondence between Novikov algebras and admissible Novikov
algebras induces the following correspondence on symmetric
bilinear forms.

\begin{prop}
Let $(\mathcal A,\star,\alpha)$ be a Hom-Novikov superalgebra and
$\mathfrak B$ be an even supersymmetric bilinear form on $\mathcal A$.
Assume that $(\mathcal A,\circ,\alpha)$ is the corresponding admissible
Hom-Novikov superalgebra, that is, the Hom-$2$-superalgebra of
$(\mathcal A,\star,\alpha)$.

Then $\mathfrak B$ is invariant on
$(\mathcal A,\circ,\alpha)$ if and only if the following condition is satisfied:
\begin{equation}\label{eq:Hom-super-bilinear-Novikov-admissible}
\mathfrak B(x\star y,\alpha(z))
=
-(-1)^{|x||y|}
\mathfrak B\Bigl(
\alpha(y),\,
x\star z
+
(-1)^{|x||z|}z\star x
\Bigr),
\end{equation}
for all homogeneous elements
$x,y,z\in \mathcal{H}(A)$.    
\end{prop}

\begin{proof}
Let $x,y,z\in \mathcal H(\mathcal A)$. Suppose that $\mathfrak B$ is invariant on
$(\mathcal A,\circ,\alpha)$. Then we have
\begin{eqnarray*}
3\mathfrak B(x\star y,\alpha(z)) &\overset{(\ref{eq:pro:admissible Novikov algebras and Novikov algebras1})}{=}&
-\mathfrak B(x\circ y,\alpha(z))
+2(-1)^{|x||y|}\mathfrak B(y\circ x,\alpha(z)) \\
&\overset{(\ref{inv-Hom-anti-pre-Lie})}{=}&
-(-1)^{|x||y|}\mathfrak B(\alpha(y),[x,z])
+2\mathfrak B(\alpha(x),[y,z]) \\
&\overset{(\ref{sup-comm-2-cocyc})}{=}&
\mathfrak B(\alpha(x),[y,z])
+(-1)^{|z|(|x|+|y|)}\mathfrak B(\alpha(z),[y,x]).
\end{eqnarray*}
Due to the supersymmetry of $x$ and $z$ in the above equation, we have
\begin{equation}\label{eq:Hom-super-bilinear}
\mathfrak B(x\star y,\alpha(z))
=(-1)^{|x|(|y|+|z|)}
\mathfrak B(z\star y,\alpha(x)),
\quad \forall x,y,z\in \mathcal H(\mathcal A).
\end{equation}

Moreover, we have
\begin{eqnarray*}
3\mathfrak B(x\star y,\alpha(z)) 
&=&
\mathfrak B(\alpha(x),[y,z])
+(-1)^{|z|(|x|+|y|)}\mathfrak B(\alpha(z),[y,x]) \\
&=&
\mathfrak B(\alpha(x),y\star z)
-(-1)^{|y||z|}\mathfrak B(\alpha(x),z\star y)
+(-1)^{|z|(|x|+|y|)}
\mathfrak B(\alpha(z),y\star x)
 \\
&&\quad-(-1)^{|z|(|x|+|y|)}
\mathfrak B(\alpha(z),x\star y) \\
&\overset{(\ref{eq:Hom-super-bilinear})}{=}&
2(-1)^{|z|(|x|+|y|)}\mathfrak B(\alpha(z),x\star y)
-(-1)^{|x||y|}
\mathfrak B(\alpha(y),x\star z) \\
&&\quad
-(-1)^{|x|(|y|+|z|)}
\mathfrak B(\alpha(y),z\star x).
\end{eqnarray*}
Hence Eq.~(\ref{eq:Hom-super-bilinear-Novikov-admissible}) holds.

Conversely, if $\mathfrak B$ satisfies
Eq.~(\ref{eq:Hom-super-bilinear-Novikov-admissible}) on the
Hom-Novikov superalgebra $(\mathcal A,\star,\alpha)$, then
Eq.~(\ref{eq:Hom-super-bilinear}) holds and
\begin{eqnarray*}
\mathfrak B(x\circ y,\alpha(z)) 
&\overset{(\ref{eq:pro:H-anti-pre-Lie superalgebras from H-Novikov superalgebras})}{=}&
\mathfrak B(x\star y,\alpha(z))
+2(-1)^{|x||y|}
\mathfrak B(y\star x,\alpha(z)) \\
&=&
(-1)^{|x||y|}
\mathfrak B(y\star x,\alpha(z))
+(-1)^{|z|(|x|+|y|)}\mathfrak B(\alpha(z),x\star y)
\\&+&(-1)^{|z|(|x|+|y|)}
\mathfrak B(\alpha(z),y\star x) \\
&\overset{(\ref{eq:Hom-super-bilinear-Novikov-admissible}),
(\ref{eq:Hom-super-bilinear})}{=}&
(-1)^{|z|(|x|+|y|)}\mathfrak B(z\star x,\alpha(y))
-(-1)^{|y||z|}
\mathfrak B(x\star z,\alpha(y)) \\
&=&
(-1)^{|z||y|}\mathfrak B([x,z],\alpha(y)).
\end{eqnarray*}

Thus $\mathfrak B$ is invariant on $(\mathcal A,\circ,\alpha)$.
\end{proof}

The correspondence between Hom-Novikov superalgebras and
infinite-dimensional Hom-Lie superalgebras gives a correspondence between admissible Hom-Novikov superalgebras and
infinite-dimensional Hom-Lie superalgebras.

\begin{prop} \label{lem:BN}
Let $\mathcal A$ be a $\mathbb Z_2$-graded vector space with a
bilinear operation $\star:\mathcal A\times \mathcal A\rightarrow \mathcal A$ and $\alpha:\mathcal A\to\mathcal A$ be an even linear map. Set $\hat
{\mathcal A}=\mathcal A\otimes {\mathbb K}[t,t^{-1}]$. Define a bilinear operation
$[\cdot,\cdot]:\hat{\mathcal A}\otimes \hat {\mathcal A}\rightarrow \hat{\mathcal A}$ by
\begin{equation}
[x\otimes t^{m+1}, y\otimes t^{n+1}]=((m+1)x\star y-(-1)^{|x||y|}(n+1)y\star
x)\otimes t^{m+n+1},\;\;\forall x,y\in \mathcal A, m,n\in {\mathbb Z}.
\end{equation}
Then $(\hat{\mathcal A},[\cdot,\cdot],\alpha)$ is a Hom-Lie superalgebra if and only if $(\mathcal A,\star,\alpha)$
is a Hom-Novikov superalgebra.
\end{prop}

Combining Propositions~\ref{pro:admissible Novikov algebras and
Novikov algebras} and \ref{lem:BN} together, we have the following
conclusion.

\begin{cor}\label{cor:BN-A}
Let $\mathcal A$ be a $\mathbb Z_2$-graded vector space with a bilinear operation
$\circ:\mathcal A\times\mathcal A\rightarrow \mathcal A$ and $\alpha:\mathcal A\to\mathcal A$ be an even linear map. Set $\hat{\mathcal A}=\mathcal A\otimes {\mathbb
K}[t,t^{-1}]$. Define a bilinear operation $[\cdot,\cdot]:\hat{\mathcal A}\times \hat {\mathcal A}\rightarrow \hat{\mathcal A}$ by
\begin{equation}
[x\otimes t^{m+1}, y\otimes t^{n+1}]=((m+2n+3)x\circ
y-(-1)^{|x||y|}(2m+n+3)y\circ x)\otimes t^{m+n+1},\;\;\forall x,y\in\mathcal H(\mathcal A), m,n\in
{\mathbb Z}.
\end{equation}
Then $(\hat{\mathcal A},[\cdot,\cdot],\alpha)$ is a Hom-Lie superalgebra if and only if $(\mathcal A,\star,\alpha)$
is an admissible Hom-Novikov superalgebra.
\end{cor}

Next we interpret admissible Hom-Novikov superalgebras in terms of
anti-super-$\mathcal O$-operators.


\begin{df}\label{defi:admissible anti O-operators}
Let $(\mathcal A,[\cdot,\cdot],\alpha)$ be a Hom-Lie superalgebra and  
$(V,\rho,\beta)$ be a representation of $(\mathcal A,[\cdot,\cdot],\alpha)$.  
Assume that $T:V\rightarrow \mathcal A$ is an  anti-super-$\mathcal O$-operator associated to the representation $(V,\rho,\beta)$. Then $T$ is called \textbf{admissible} if
\begin{equation}\label{eq:Hom-super-admissible-anti-O-operator}
\begin{aligned}
2\,\rho(T(\beta(u)))\rho(T(v))w
-2(-1)^{|v||w|}
\rho(T(\beta(u)))\rho(T(w))v
&=\rho\Big(T(\rho(T(u))v)\Big)\beta(w)
\\&-(-1)^{|v||w|}
\rho\Big(T(\rho(T(u))w)\Big)\beta(v),
\end{aligned}
\end{equation}
for all homogeneous $u,v,w\in V$.

In particular, an anti-Rota-Baxter operator
$\mathcal R:\mathcal A\rightarrow\mathcal A$ on the Hom-Lie superalgebra
$(\mathcal A,[\cdot,\cdot],\alpha)$ is called \textbf{admissible} if
\begin{equation}\label{eq:Hom-super-admissible-anti-RB}
\begin{aligned}
2[\mathcal R(\alpha(x)),[\mathcal R(y),z]]
-2(-1)^{|y||z|}[\mathcal R(\alpha(x)),[\mathcal R(z),y]]
&=[\mathcal R([\mathcal R(x),y]),\alpha(z)]
\\&-(-1)^{|y||z|}
[\mathcal R([\mathcal R(x),z]),\alpha(y)],
\end{aligned}
\end{equation}
for all homogeneous $x,y,z\in\mathcal A$.    
\end{df}
\begin{prop}\label{pro:admissible operator}
Let $T:V\rightarrow \mathcal A$
be an admissible anti-super-$\mathcal O$-operator on a Hom-Lie superalgebra $(\mathcal A,[\cdot,\cdot],\alpha)$ associated to a representation 
$(V,\rho,\beta)$. Then $T$ is strong.

Moreover,
$(V,\circ,\beta)$ is an admissible Hom-Novikov superalgebra, where the
operation $\circ$ is given by Eq.~(\ref{ant-O-oper-to-anti-pre-Lie}).    
\end{prop}

\begin{proof}
 Let $u,v,w\in\mathcal H(V)$.
By Eqs.~(\ref{cond-ant-O-oper-Hom-Lie2}) and
(\ref{eq:Hom-super-admissible-anti-O-operator}), we have
\begin{eqnarray*}
&&2\rho([T(u),T(v)])\beta(w)
+2(-1)^{|u|(|v|+|w|)}
\rho([T(v),T(w)])\beta(u)\\
&&\hspace{1cm}
+2(-1)^{|w|(|u|+|v|)}
\rho([T(w),T(u)])\beta(v)\\
&&=
\rho(T(\rho(T(u))v))\beta(w)
-(-1)^{|v||w|}
\rho(T(\rho(T(u))w))\beta(v)\\
&&\hspace{0.4cm}
+(-1)^{|u|(|v|+|w|)}
\rho(T(\rho(T(v))w))\beta(u)\\
&&\hspace{0.4cm}
-(-1)^{|u||v|}
\rho(T(\rho(T(v))u))\beta(w)\\
&&\hspace{0.4cm}
+(-1)^{|w|(|u|+|v|)}
\rho(T(\rho(T(w))u))\beta(v)\\
&&\hspace{0.4cm}
-(-1)^{|u|(|v|+|w|)+|v||w|}
\rho(T(\rho(T(w))v))\beta(u)\\
&&=(-1)^{|u||v|}
\rho([T(v),T(u)])\beta(w)
+(-1)^{|u|(|v|+|w|)+|v||w|}
\rho([T(w),T(v)])\beta(u)\\
&&\hspace{0.4cm}
+(-1)^{|v||w|}
\rho([T(u),T(w)])\beta(v).
\end{eqnarray*}
Hence Eq.~(\ref{cond-strong-O-Oper})
holds and thus $T$ is strong.

Moreover, we have
\begin{eqnarray*}
2\beta(u)\circ[v,w]
&=&
2\rho(T(\beta(u)))\rho(T(v))w
-2(-1)^{|v||w|}
\rho(T(\beta(u)))\rho(T(w))v\\
&=&
\rho(T(\rho(T(u))v))\beta(w)
-(-1)^{|v||w|}
\rho(T(\rho(T(u))w))\beta(v)\\
&=&
(u\circ v)\circ\beta(w)
-(-1)^{|v||w|}
(u\circ w)\circ\beta(v).
\end{eqnarray*}
Then by Theorem \ref{Hom-ant-pre-Lie-by-ant-O-oper},
$(V,\circ,\beta)$ is an admissible Hom-Novikov superalgebra.   
\end{proof}

\begin{cor}\label{H-Lie-To-H-Nov}
Let $(\mathcal A,[\cdot,\cdot],\alpha)$ be a Hom-Lie superalgebra, and
$\mathcal R:\mathcal A\rightarrow \mathcal A$ be an admissible
anti-Rota-Baxter operator on $(\mathcal A,[\cdot,\cdot],\alpha)$. Then $(\mathcal A,\circ,\alpha)$ given by
Eq.~(\ref{ant-R-B-to-anti-pre-Lie}) defines an admissible Hom-Novikov superalgebra.
Conversely, if $\mathcal R:\mathcal A\rightarrow \mathcal A$ is a
linear transformation on a Hom-Lie superalgebra $(\mathcal A,[\cdot,\cdot],\alpha)$ such
that Eq.~(\ref{ant-R-B-to-anti-pre-Lie}) defines an admissible Hom-Novikov superalgebra,
then $\mathcal R$ satisfies Eqs.~(\ref{Cor-R-B-Hom-anti-pre-Lie}) and
(\ref{eq:Hom-super-admissible-anti-RB}).
\end{cor}
\begin{proof}
The first half part follows from Proposition \ref{pro:admissible
operator} by letting $\rho=\mathfrak{ad}$. The second half part
follows from Definition~\ref{defi:admissible Hom-Novikov superalgebras}
directly.
\end{proof}

\begin{exa}
Consider the $3$-dimensional Hom-Lie superalgebra $(\mathcal{A}, [\cdot, \cdot], \alpha)$ introduced in Example~\ref{Ex-H-Lie-sup}. Define an even linear operator $\mathcal{R} \in \mathrm{End}_{\bar{0}}(\mathcal{A})$ on the basis elements $\{e_1, e_2, e_3\}$ by
$$
\mathcal{R}(e_1) = e_1, \quad \mathcal{R}(e_2) = e_3, \quad \text{and} \quad \mathcal{R}(e_3) = 0.
$$
A direct calculation verifies that $\mathcal{R}$ is an admissible anti-Rota-Baxter operator on $\mathcal{A}$. Consequently, applying Corollary~\ref{H-Lie-To-H-Nov} yields an induced admissible Hom-Novikov superalgebra structure $(\mathcal{A}, \circ, \alpha)$, whose multiplication is non-zero on the basis element product
$$
e_1 \circ e_2 = -[\mathcal{R}(e_1), e_2] = -e_3.
$$
\end{exa}

\begin{cor}\label{cor:admissible anti O-operator}
Let $(\mathcal A,[\cdot,\cdot],\alpha)$ be a Hom-Lie superalgebra. Then there is a compatible
admissible Hom-Novikov superalgebra structure on $\mathcal A$ if and only if
there exists an invertible admissible anti-super-$\mathcal{O}$-operator
on $(\mathcal A,[\cdot,\cdot],\alpha)$.
\end{cor}

\begin{proof}
Let $x,y,z\in\mathcal A$. Suppose that $(\mathcal A,\circ,\alpha)$ is a compatible
admissible Novikov algebra structure on $(\mathcal A,[\cdot,\cdot],\alpha)$. Then by
Eq.~(\ref{admiss-Hom-Nov-sup2}), we have
$$2\mathfrak{L}_{\circ}(\alpha(x))\mathfrak{L}_{\circ}(y)z-(-1)^{|y||z|}2\mathfrak{L}_{\circ}(\alpha(x))\mathfrak{L}_{\circ}(z)y=\mathfrak{L}_{\circ}(\mathfrak{L}_{\circ}(x)y)\alpha(z)-(-1)^{|y||z|}\mathfrak{L}_{\circ}(\mathfrak{L}_{\circ}(x)z)\alpha(y).$$
By Corollary \ref{cor-comp-h-anti-pre-Lie-inv-O-op},
$T=\mathrm{id}$ is an invertible admissible
anti-super-$\mathcal{O}$-operator on $(\mathcal A,[\cdot,\cdot],\alpha)$ associated to
$(\mathcal A,-\mathfrak{L}_{\circ},\alpha)$. Conversely, suppose that
$T:V\rightarrow \mathcal A$ is an invertible admissible
anti-super-$\mathcal{O}$-operator on $(\mathcal A,[\cdot,\cdot],\alpha)$ associated to
$(V,\rho,\beta)$. Let $x\circ_{\mathcal A} y=-T(\rho(x)T^{-1}(y))$. Then
\begin{eqnarray*}
2\alpha(x)\circ_{\mathcal A}[y,z]
&=&2T\rho(\alpha(x))\rho(y)T^{-1}(z)-(-1)^{|y||z|}2T\rho(\alpha(x))\rho(z)T^{-1}(y)\\
&=&T(\rho(T(\rho(x)T^{-1}(y)))T^{-1}(\alpha(z)))-(-1)^{|y||z|}T(\rho(T(\rho(x)T^{-1}(z)))T^{-1}(\alpha(y)))\\
&=&(x\circ_{\mathcal A} y)\circ_{\mathcal A}\alpha(z)-(-1)^{|y||z|}(x\circ_{\mathcal A} z)\circ_{\mathcal A}\alpha(y).
\end{eqnarray*}

By Corollary \ref{cor-comp-h-anti-pre-Lie-inv-O-op},
$(\mathcal A,\circ_{\mathcal A},\alpha)$ is a compatible admissible Hom-Novikov superalgebra
structure on $(\mathcal A,[\cdot,\cdot],\alpha)$.
\end{proof}

\begin{prop}
Let $\mathfrak B$ be a nondegenerate super-commutative $2$-cocycle on a
Hom-Lie superalgebra $(\mathcal A,[\cdot,\cdot],\alpha)$. Define a bilinear operation
$\circ: {\mathcal A}\times \mathcal A\rightarrow \mathcal A$ by Eq.~(\ref{inv-Hom-anti-pre-Lie}). Then $(\mathcal A,\circ,\alpha)$ is an
admissible Hom-Novikov superalgebra if and only if the following equation
holds:

\begin{equation}\label{eq:addition}
2\mathfrak B([x,y],[z,t])=(-1)^{|y||z|}\mathfrak B([x\circ z,
y],\alpha(t))-(-1)^{|t|(|y|+|z|)}\mathfrak B([x\circ t,y],\alpha(z))),\;\;\forall x,y,z,t\in\mathcal H( \mathcal A).
\end{equation}
\end{prop}

\begin{proof}
By Theorem \ref{ant-h-pr-Lie-from-nondeg-bil}, $(\mathcal A,\circ,\alpha)$ is a Hom-anti-pre-Lie superalgebra. Moreover, for
all $x,y,z,t\in\mathcal H(\mathcal A)$,
\begin{eqnarray*}
&&2\mathfrak B([x,y],[z,t])+(-1)^{|z|(|x|+|y|)+|y||t|}\mathfrak B(\alpha(z),[x\circ t,y])-(-1)^{|t|(|x|+|y|+|z|)+|y||z|}\mathfrak B(\alpha(t),[x\circ z,
y])\\
&&\quad=(-1)^{|x||y|}\mathfrak B(\alpha(y),2\alpha(x)\circ[z,t]-(-1)^{|z||t|}(x\circ t)\circ\alpha(z)+(x\circ z)\circ \alpha(t)).
\end{eqnarray*}
Therefore $(\mathcal A,\circ,\alpha)$ is an admissible Hom-Novikov superalgebra if and only
if Eq.~(\ref{eq:addition}) holds.
\end{proof}

\begin{rem} The above conclusion can be obtained from another
approach. Let $T:\mathcal A\rightarrow \mathcal A^{*}$ be an even linear map
defined by 
\begin{equation}\label{eq:HomT} \langle T(x),y\rangle=\mathfrak B(x,y),\qquad \forall x,y\in \mathcal H(\mathcal A). 
\end{equation}
Then it is straightforward to show that
$T^{-1}$ is an admissible anti-$\mathcal{O}$-operator of $(\frak
g,[\cdot,\cdot],\alpha)$ associated to $(\mathcal A^{*},\mathrm{ad}^{\star},(\alpha^{-1})^*)$ if and
only if Eq.~(\ref{eq:addition}) holds. Thus the same conclusion
follows.
\end{rem}

\subsection{Constructions from super-commutative Hom-associative superalgebras}

We introduce the notion of admissible pairs on super-commutative
Hom-associative superalgebras as a generalization  of derivations. Then we
give the constructions of Hom-Novikov superalgebras and the corresponding
admissible Hom-Novikov superalgebras from super-commutative Hom-associative superalgebras
with admissible pairs, generalizing the known various
constructions of Hom-Novikov superalgebras from super-commutative Hom-associative
superalgebras with derivations. Especially under certain conditions,
they coincide with the admissible Hom-Novikov superalgebra structures
induced from the natural nondegenerate  super-commutative $2$-cocycles on
the Hom-Lie superalgebras obtained from the super-commutative Hom-associative
superalgebras with derivations.

\begin{df}
An \textbf{admissible} pair on  a super-commutative Hom-associative superalgebra $(\mathcal A,\mu,\alpha)$
consists of two even linear maps
$\mathcal P,\mathcal Q:\mathcal A\rightarrow \mathcal A$ satisfying 

\begin{eqnarray}
&&\mathcal P\alpha=\alpha\mathcal P,\;\;\mathcal Q\alpha=\alpha\mathcal Q\label{eq:defi:admissible pair1}\\
&&\mathcal Q(\mu(x,y))=\mu(\mathcal Q(x),y)+\mu(x,\mathcal P(y)), \;\;\forall x,y\in\mathcal H(\mathcal A)\label{eq:defi:admissible pair2}.
\end{eqnarray}
We denote it by $(\mathcal P,\mathcal Q)$.
\end{df}

\begin{exa}
Let $(A,\mu,\alpha)$ be a super-commutative Hom-associative superalgebra.
\begin{enumerate}
\item Let $\mathfrak D$ be an even derivation on $(\mathcal A,\mu,\alpha)$. Then
\begin{enumerate}
\item $(\mathfrak D,\mathfrak D)$ is an admissible pair. 
\item $(\mathfrak D,\mathfrak D+\lambda {\rm
Id})$ is an admissible pair, where $\lambda\in \mathbb K$.
\item
$(\mathfrak D,\mathfrak D+\mathfrak L_\mu (a))$ is an admissible pair, where $a\in
\mathcal H(\mathcal A)$.  
\item Let $\pi:\mathcal A\rightarrow \mathcal A$ be a linear map. We say that $\pi$ is called {\bf admissible to $(\mathcal A,\mathfrak D)$} if
$$\mu(\pi(x), y)=\pi(\mu(x, y))+\mu(x,\mathfrak D(y)),\;\;\forall x,y\in \mathcal H(\mathcal A).$$
In this setting $(\mathfrak D,-\pi)$ is an admissible pair.
 \item Let $\mathfrak B$ be a nondegenerate
supersymmetric invariant bilinear form on $(\mathcal A,\mu,\alpha)$.  Let $\hat{\mathfrak D}:\mathcal A\rightarrow \mathcal A$ be the adjoint operator of
$\mathfrak D$ with respect to $\mathfrak B$, that is,
$$\mathfrak B(\mathfrak D(x),y)=\mathfrak B(x,\hat{\mathfrak D}(y)),\forall x,y\in\mathcal H(\mathcal A).$$
Then $\hat{\mathfrak D}$ is admissible to $(\mathcal A,\mathfrak D)$, or
equivalently, $(\mathfrak D,-\hat{\mathfrak D})$ is an admissible pair.
\end{enumerate}
\item Let $\mathcal Q:\mathcal A\rightarrow \mathcal A$ be an even linear map satisfying $\mathcal Q\alpha=\alpha\mathcal Q$ and 
$$\mathcal Q(\mu(x,y))=\mu(\mathcal Q(x), y+\mu(x,\mathcal Q(y)+\lambda \mu(x,y),\;\;\forall
x,y\in\mathcal H(\mathcal A),$$ where $\lambda\in \mathbb K$. Then $(\mathcal Q+\lambda {\rm
Id},\mathcal Q)$ is an admissible pair. 
\item Let $\mathcal Q:\mathcal A\rightarrow \mathcal A$ be an even
linear map satisfying $\mathcal Q\alpha=\alpha\mathcal Q$ and 
$$\mathcal Q(\mu(x, y))=\mu(\mathcal Q(x),y)+\mu(x,\mathcal Q(y))+\mu(a,\mu(x,y)),\;\;\forall
x,y\in\mathcal H(\mathcal A),$$ where $a\in\mathcal H(\mathcal A)$. Then $(\mathcal Q+\mathfrak L_\mu (a),\mathcal Q)$ is
an admissible pair.
\end{enumerate}

\end{exa}

The following proposition gives a construction of a Hom-Novikov superalgebra structure using an
admissible pair.
\begin{prop}\label{ex:Novikov algebra from admissible pair}
Let $(\mathcal P,\mathcal Q)$ be an admissible pair on a super-commutative Hom-associative
superalgebra $(\mathcal A,\mu,\alpha)$. Define a bilinear operation $\star$ on $\mathcal A$ by
\begin{equation}\label{eq:ex:Novikov algebra from admissible pair}
x\star y=\mu(x,\mathcal Q(y)), \;\;\forall x,y\in\mathcal H(\mathcal A).
\end{equation}
Then $(\mathcal A,\star,\alpha)$ is a Hom-Novikov superalgebra.
\end{prop}

\begin{proof}
Let $x,y,z\in\mathcal H(\mathcal A)$. Then we have
\begin{eqnarray*}
(x\star y)\star\alpha(z)-\alpha(x)\star(y\star z)&=&\mu(\mu(x,\mathcal Q(y)),\mathcal Q(\alpha(z)))-\mu(\alpha(x),\mathcal Q(\mu(y,\mathcal Q(z)))\\&=&
\mu(\alpha(x),\mu(\mathcal Q(y),\mathcal Q(z)))-\mu(\alpha(x),\mu(\mathcal Q(y),\mathcal Q(z)))\\&&-\mu(\alpha(x),\mu(y,\mathcal P(\mathcal Q(z))))
\\&=&-\mu(x,\mu(y,\mathcal P(\mathcal Q(z)))\\&=&(y\star x)\star \alpha(z)-(-1)^{|x||y|}\alpha(y)\star
(x\star z).
\end{eqnarray*}
Moreover, we have
\begin{eqnarray*}
(x\star y)\star\alpha(z)&=&\mu(\mu(x,\mathcal Q(y)),\mathcal Q(\alpha(z)))\\&=&(-1)^{|y||z|}\mu(\alpha(x),\mu(\mathcal Q(z),\mathcal Q(y)))\\&=&(-1)^{|y||z|}\mu(\mu(x,\mathcal Q(z)),\mathcal Q(\alpha(y)))\\&=&(-1)^{|y||z|}(x\star z)\star\alpha(y).
\end{eqnarray*}
Hence $(\mathcal A,\star,\alpha)$ is a Hom-Novikov superalgebra.
\end{proof}

\begin{cor}\label{cor:admissible Novikov algebra from admissible pair}
Let $(\mathcal P,\mathcal Q)$ be an admissible pair on a super-commutative Hom-associative
superalgebra $(\mathcal A,\mu,\alpha)$. Define a bilinear operation $\circ$ on $\mathcal A$ by
\begin{equation}\label{eq:cor:admissible Novikov algebra from admissible pair}
x\circ y=\mu(x,\mathcal Q(y))+2\mu(\mathcal Q(x),y),\;\;\forall x,y\in\mathcal H(\mathcal  A).
\end{equation}
Then $(\mathcal A,\circ,\alpha)$ is an admissible Hom-Novikov superalgebra. 
\end{cor}

\begin{exa}\label{ex:der}
Let $\mathfrak D$ be an even derivation on a super-commutative Hom-associative superalgebra
$(\mathcal A,\mu,\alpha)$. Define a bilinear operation $\circ:\mathcal A\times
\mathcal A\rightarrow \mathcal A$ by
\begin{equation}
x\circ y=\mu(x,\mathfrak D(y))+2\mu(\mathfrak D(x),y)+\mu(a,\mu(x, y)),\;\;\forall
x,y\in\mathcal H(\mathcal A), \end{equation} where $a\in {\mathbb K}$ or $a\in\mathcal H(\mathcal A)$.
Then $(\mathcal A,\circ,\alpha)$ is an admissible Hom-Novikov superalgebra. Note that in
this case, both $(\mathfrak D,\mathfrak D+\frac{1}{3}a {\rm Id})$ ($a\in \mathbb K$)
and $(\mathfrak D,\mathfrak D+\frac{1}{3}\mathfrak L_\mu(a))$ ($a\in\mathcal H(\mathcal A)$) are
admissible pairs.
\end{exa}

\begin{prop}\label{ex:commutative 2-cocycle}
Let $(\mathcal P,\mathcal Q)$ be an admissible pair on a super-commutative Hom-associative
superalgebra $(\mathcal A,\mu,\alpha)$. Define a bilinear operation $[\cdot,\cdot]:\mathcal A\times
\mathcal A\rightarrow\mathcal A$ by
\begin{equation}\label{eq:ex:commutative 2-cocycle}
[x,y]=\mu(\mathcal Q(x),y)-\mu(x,\mathcal Q(y))=\mu(\mathcal P(x),y)-\mu(x,\mathcal P(y)), \forall x,y\in\mathcal H(\mathcal A).
\end{equation}
Then $(\mathcal A,[\cdot,\cdot],\alpha)$ is a Hom-Lie superalgebra. Moreover, if there is a symmetric bilinear form $\mathfrak{B}$ on $(\mathcal A,\mu,\alpha)$ satisfying
\begin{eqnarray}
 \mathfrak{B}(\alpha(x),\alpha(y))&=&\mathfrak{B}(x,y),\label{def-sym-inv-bil-form-H-ass1} \\
 \mathfrak{B}(\mu(x,y),\alpha(z))&=&\mathfrak{B}(\alpha(x),\mu(y,z)),\label{def-sym-inv-bil-form-H-ass2}
\end{eqnarray}
for any $x,y,z\in\mathcal H(\mathcal A)$,
then $\mathfrak{B}$ is a super-commutative $2$-cocycle on $(\mathcal A,[\cdot,\cdot],\alpha)$.
\end{prop}

\begin{proof}
The first half part is obtained by a direct proof or follows from
Proposition~\ref{ex:Novikov algebra from admissible pair} or
Corollary~\ref{cor:admissible Novikov algebra from admissible
pair} since $(\mathcal A,[\cdot,\cdot],\alpha)$ is the sub-adjacent Lie algebra of an
admissible Hom-Novikov superalgebra. Note that the second equality in Eq.~(\ref{eq:ex:commutative 2-cocycle}) is due to
Eq.~(\ref{eq:defi:admissible pair2}) and the fact that $(\mathcal A,\mu,\alpha)$ is super-commutative.
For the second half part, let
$x,y,z\in\mathcal H(\mathcal A)$. Then we have
\begin{eqnarray*}
\mathfrak{B}([x,y],\alpha(z))&=&\mathfrak{B}(\mu(\mathcal Q(x),y),\alpha(z))-\mathfrak{B}(\mu(x,\mathcal Q(y)),\alpha(z))\\&=&\mathfrak{B}(\mathcal Q(\alpha(x)),\mu(y,z))-(-1)^{|x||y|}\mathfrak{B}(\mathcal Q(\alpha(y)),\mu(x,z)),
\end{eqnarray*}
\begin{eqnarray*}
\mathfrak{B}([y,z],\alpha(x))&=&\mathfrak{B}(\mu(\mathcal Q(y),z),\alpha(x))-\mathfrak{B}(\mu(y,\mathcal Q(z)),\alpha(x))\\&=&\mathfrak{B}(\mathcal Q(\alpha(y)),\mu(z,x))-(-1)^{|y||z|}\mathfrak{B}(\mathcal Q(\alpha(z)),\mu(y, x)),
\end{eqnarray*} and 
\begin{eqnarray*}
\mathfrak{B}([z,x],\alpha(y))&=&\mathfrak{B}(\mu(\mathcal Q(z),x),\alpha(y))-\mathfrak{B}(\mu(z,\mathcal Q(x)),\alpha(y))\\&=&\mathfrak{B}(\mathcal Q(\alpha(z)),\mu(x,y))-(-1)^{|x||z|}\mathfrak{B}(\mathcal Q(\alpha(x)),\mu(z,y)).
\end{eqnarray*}
Thus, by a direct computation, $\mathfrak{B}$ is a super-commutative $2$-cocycle on $(\mathcal A,[\cdot,\cdot],\alpha)$.
\end{proof}

In particular, as a special case, we get the following conclusion.

\begin{cor}\label{cor:construction}
Let $\mathfrak D$ be an even derivation on a super-commutative Hom-associative superalgebra
$(\mathcal A,\mu,\alpha)$. Then Eq.~(\ref{eq:ex:commutative 2-cocycle}) defines
a Hom-Lie superalgebra $(\mathcal A,[\cdot,\cdot],\alpha)$. If $\mathfrak B$ is a supersymmetric
invariant bilinear form on $(\mathcal A,\mu,\alpha)$, then $\mathfrak B$ is a
super-commutative $2$-cocycle on $(\mathcal A,[\cdot,\cdot],\alpha)$.
\end{cor}

\subsection{Hom-anti-pre-Lie Poisson superalgebras and Hom-Novikov-Poisson superalgebras}

We extend the correspondence between Hom-Novikov superalgebras and
admissible Hom-Novikov superalgebras to the level of Poisson type
structures, and hence introduce the notions of Hom-anti-pre-Lie Poisson
superalgebras and admissible Hom-Novikov-Poisson superalgebras. The relationships with transposed Hom-Poisson superalgebras and
Hom-Novikov-Poisson superalgebras as well as a tensor theory are given.

\begin{df}\label{defi:anti-pre-Lie Poisson}
A \textbf{Hom-anti-pre-Lie Poisson superalgebra} is a quadruple
$(\mathcal A,\mu,\circ,\alpha)$, where $(\mathcal A,\mu,\alpha)$ is a super-commutative
Hom-associative superalgebra and $(\mathcal A,\circ,\alpha)$ is a Hom-anti-pre-Lie superalgebra satisfying the following conditions:
\begin{equation}\label{eq:defi:H-anti-pre-Lie Poisson1}
2\mu((x\circ y),\alpha(z))-2(-1)^{|x||y|}\mu((y\circ x), \alpha(z))=(-1)^{|x||y|}\mu(\alpha(y),(x\circ z))-\mu(\alpha(x),(y\circ z)),
\end{equation}
\begin{equation}\label{eq:defi:H-anti-pre-Lie Poisson2}
2\alpha(x)\circ\mu(y, z)=(-1)^{|z|(|x|+|y|)}\mu(z, x)\circ \alpha(y)+(-1)^{|z|(|x|+|y|)}\mu(\alpha(z),(x\circ y)),
\end{equation}
for all $x,y,z\in\mathcal H(\mathcal A)$.\\
An \textbf{admissible Hom-Novikov-Poisson superalgebra} is a Hom-anti-pre-Lie
Poisson superalgebra $(\mathcal A,\mu,\circ,\alpha)$ such that $(\mathcal A,\circ,\alpha)$ is
an admissible Hom-Novikov superalgebra.
\end{df}

\begin{df}\label{defi:Novikov Poisson algebra}
A \textbf{Hom-Novikov-Poisson superalgebra} is a quadruple $(\mathcal A,\mu,\star,\alpha)$, where $(\mathcal A,\mu,\alpha)$ is a super-commutative Hom-associative superalgebra and $(\mathcal A,\star,\alpha)$ is a Hom-Novikov superalgebra satisfying the following conditions:
\begin{equation}\label{eq:defi:H-Novikov Poisson superalgebra1}
\mu(x, y)\star\alpha(z)=\mu(\alpha(x),(y\star z)),
\end{equation}
\begin{equation}\label{eq:defi:H-Novikov Poisson superalgebra2}
\mu((x\star y),\alpha(z))-(-1)^{|x||y|}\mu((y\star x),\alpha(z))=\alpha(x)\star\mu(y,z)-(-1)^{|x||y|}\alpha(y)\star\mu(x,z),
\end{equation}
for all $x,y,z\in\mathcal H(\mathcal A)$.
\end{df}

\begin{lem} \label{lem:H-anti-pre-Lie Poisson}
Let $(\mathcal A,\mu,\alpha)$ be a super-commutative Hom-associative superalgebra, and
$\circ:\mathcal A\times \mathcal A\rightarrow \mathcal A$ be an even bilinear operation. Suppose that
Eqs.~(\ref{eq:defi:H-anti-pre-Lie Poisson1}) and
(\ref{eq:defi:H-anti-pre-Lie Poisson2}) hold.
Then, the following equations hold:
\begin{eqnarray}\label{eq:defi:H-anti-pre-Lie Poisson3}
&&2\alpha(x)\circ\mu(y,z)=(-1)^{|x||y|}\mu(y,x)\circ\alpha(z)+\mu(\alpha(y),(x\circ z)),\\
\label{eq:defi:H-anti-pre-Lie Poisson4}
&&(-1)^{|z|(|x|+|y|)}\mu(z,x)\circ\alpha(y)+(-1)^{|z|(|x|+|y|)}\mu(\alpha(z),(x\circ y))=(-1)^{|x||y|}\mu(y,x)\circ\alpha(z)\\&&\qquad\qquad\qquad\qquad\qquad\qquad\qquad\qquad\qquad\qquad\qquad\qquad\quad+(-1)^{|x||y|}\mu(\alpha(y),(x\circ z)),\nonumber\\
\label{eq:defi:H-anti-pre-Lie Poisson5}
&&\mu(\alpha(x),(y\circ z))-(-1)^{|x||y|}\mu(\alpha(y),(x\circ z))=2(-1)^{|x||y|}\alpha(y)\circ\mu(x, z)-2\alpha(x)\circ\mu(y,z),\\
\label{eq:defi:H-anti-pre-Lie Poisson6}
&&\mu(\alpha(x),(y\circ z))-(-1)^{|x||y|}\mu(\alpha(y),(x\circ z))=2(-1)^{|y||z|}\mu(\alpha(x),(z\circ y))\\&&\qquad\quad\qquad\qquad\qquad\qquad\qquad\qquad\qquad\qquad-(-1)^{|x|(|y|+|z|)}2\mu(\alpha(y),(z\circ x)),\nonumber\\
\label{eq:defi:H-anti-pre-Lie Poisson7}
&&\mu(\alpha(x),(z\circ y))-(-1)^{|x|(|y|+|z|)}\mu(\alpha(y),(z\circ x))=(-1)^{|x||y|}\alpha(y)\circ\mu(x,z)-x\circ\mu(y,z),\\
\label{eq:defi:H-anti-pre-Lie Poisson8}
&&\alpha(x)\circ\mu(y,z)+(-1)^{|x|(|y|+|z|)}\mu(y,z)\circ\alpha(x)=(-1)^{|x||y|}\alpha(y)\circ\mu(x,z)\\&&\qquad\qquad\qquad\qquad\qquad\qquad\qquad\qquad\qquad\qquad+(-1)^{|y||z|}\mu(x,z)\circ\alpha(y),\nonumber\\
\label{eq:defi:H-anti-pre-Lie Poisson9}
&&\mu(z,(x\circ y))-(-1)^{|x||y|}\mu(\alpha(z),(y\circ x))=(-1)^{|z|(|x|+|y|)}\alpha(x)\circ\mu(y,z)\\&&\qquad\qquad\qquad\qquad\qquad\qquad\qquad\qquad\qquad-(-1)^{|z|(|x|+|y|)+|x||y|}\alpha(y)\circ\mu(x,z),\nonumber
\end{eqnarray}
for all $x,y,z\in\mathcal H(\mathcal A)$.
\end{lem}

\begin{proof}
Let $x,y,z\in\mathcal H(\mathcal A)$. Then we have
\begin{enumerate}
\item Eq.~(\ref{eq:defi:H-anti-pre-Lie Poisson3}) holds by exchanging $y$ and $z$ in Eq.~(\ref{eq:defi:H-anti-pre-Lie Poisson2}).

\item Eq.~(\ref{eq:defi:H-anti-pre-Lie Poisson4}) follows from Eqs.~(\ref{eq:defi:H-anti-pre-Lie Poisson2}) and (\ref{eq:defi:H-anti-pre-Lie Poisson3}).

\item Eq.~(\ref{eq:defi:H-anti-pre-Lie Poisson5}) follows from Eq.~(\ref{eq:defi:H-anti-pre-Lie Poisson3}).

\item Eq.~(\ref{eq:defi:H-anti-pre-Lie Poisson6}) holds since
\begin{eqnarray*}
(-1)^{|x||y|}\mu(\alpha(y),(x\circ z))-\mu(\alpha(x),(y\circ z))&\overset{(\ref{eq:defi:H-anti-pre-Lie Poisson1})}{=}&2\mu((x\circ y),\alpha(z))-2(-1)^{|x||y|}\mu((y\circ x),\alpha(z))\\
&\overset{(\ref{eq:defi:H-anti-pre-Lie Poisson1})}{=}&4(-1)^{|x||y|}\mu(\alpha(y),(x\circ z))-4(-1)^{|x|(|y|+|z|)}\mu(\alpha(y),(z\circ x))\\&&+2(-1)^{|y||z|}\mu(\alpha(x),(z\circ y))-4\mu(\alpha(x),(y\circ z))\\&&+4(-1)^{|y||z|}\mu(\alpha(x),(z\circ y))-2(-1)^{|x|(|y|+|z|)}\mu(\alpha(y),(z\circ x))\\
&=&4(-1)^{|x||y|}\mu(\alpha(y),(x\circ z))-4\mu(\alpha(x),(y\circ z))\\&&-6(-1)^{|x|(|y|+|z|)}\mu(\alpha(y),(z\circ x))+6(-1)^{|y||z|}\mu(\alpha(x),(z\circ y)).
\end{eqnarray*}

\item
Eq.~(\ref{eq:defi:H-anti-pre-Lie Poisson7}) follows from
Eqs.~(\ref{eq:defi:H-anti-pre-Lie Poisson5}) and
(\ref{eq:defi:H-anti-pre-Lie Poisson6}).

\item Eq.~(\ref{eq:defi:H-anti-pre-Lie Poisson8}) holds  since
\begin{eqnarray*}
(-1)^{|x|(|y|+|z|)}\mu(y,z)\circ\alpha(x)-(-1)^{|z|(|x|+|y|)}\mu(z,x)\circ \alpha(y)&\overset{(\ref{eq:defi:H-anti-pre-Lie Poisson4})}{=}&(-1)^{|y||z|}\mu(\alpha(x),(z\circ y))\\&&-(-1)^{|x|(|y|+|z|)}\mu(\alpha(y),(z\circ x))
\\&\overset{(\ref{eq:defi:H-anti-pre-Lie Poisson7})}{=}&(-1)^{|x||y|}\alpha(y)\circ\mu(x,
z)-\alpha(x)\circ\mu(y,z). 
\end{eqnarray*}
\item Eq.~(\ref{eq:defi:H-anti-pre-Lie Poisson9}) holds since
\small{\begin{eqnarray*}
&&\alpha(x)\circ\mu(y,z)-(-1)^{|x|(|y|+|z|)}\alpha(y)\circ\mu(z,x)-(-1)^{|z|(|x|+|y|)}\mu(\alpha(z),(x\circ y))+(-1)^{|x|(|y|+|z|)+|y||z|}\mu(\alpha(z),(y\circ x))
\\&&\overset{(\ref{eq:defi:H-anti-pre-Lie Poisson3})}{=}\alpha(x)\circ\mu(y,z)-(-1)^{|x|(|y|+|z|)}\alpha(y)\circ\mu(z, x)+(-1)^{|z|(|x|+|y|)}\mu(z,x)\circ\alpha(y)-2(-1)^{|y||z|}\alpha(x)\circ\mu(z,y)\\&&-(-1)^{|x|(|y|+|z|)+|y||z|}\mu(z,y)\circ \alpha(x)+2(-1)^{|x|(|y|+|z|)}\alpha(y)\circ\mu(z,x)\\
&&=(-1)^{|x|(|y|+|z|)}\alpha(y)\circ\mu(z,x)+(-1)^{|z|(|x|+|y|)}\mu(z,x)\circ\alpha(y)-\alpha(x)\circ\mu(y,z)-(-1)^{|x|(|y|+|z|)+|y||z|}\mu(z,y)\circ \alpha(x)\\&&\overset{(\ref{eq:defi:H-anti-pre-Lie Poisson8})}{=}0.
\end{eqnarray*}}
\end{enumerate}
Thus the conclusion holds.
\end{proof}

The following conclusion extends the correspondence between
Hom-Novikov superalgebras and admissible Hom-Novikov superalgebras to the level of
Poisson type structures.

\begin{prop}\label{pro:from admissible Novikov Poisson to Novikov Poisson}
Let $(\mathcal A,\mu,\alpha)$ be a super-commutative Hom-associative superalgebra and $\circ:\mathcal A\times\mathcal A\rightarrow\mathcal A$ be an even bilinear operation.
Let $(\mathcal A,\star,\alpha)$ be the
$(-2)$-Hom-superalgebra of $(\mathcal A,\circ,\alpha)$. Then
Eqs.~(\ref{eq:defi:H-anti-pre-Lie Poisson1}) and
(\ref{eq:defi:H-anti-pre-Lie Poisson2}) hold if and only if
Eqs.~(\ref{eq:defi:H-Novikov Poisson superalgebra1}) and
(\ref{eq:defi:H-Novikov Poisson superalgebra2}) hold. In particular,
$(\mathcal A,\mu,\circ,\alpha)$ is an admissible Hom-Novikov-Poisson superalgebra if and only if
$(\mathcal A,\mu,\star,\alpha)$ is a Novikov-Hom-Poisson superalgebra.
\end{prop}

\begin{proof}Suppose that Eqs.~(\ref{eq:defi:H-anti-pre-Lie Poisson1}) and
(\ref{eq:defi:H-anti-pre-Lie Poisson2}) hold. Then for all $x,y,z\in\mathcal H(\mathcal A)$, we have
\begin{align*}
\mu(x,y)\star\alpha(z)-\mu(\alpha(x),(y\star z))&=\mu(x,y)\circ\alpha(z)-2(-1)^{|z|(|x|+|y|)}\alpha(z)\circ\mu(x,y)-\mu(\alpha(x),(y\circ z))\\&+2(-1)^{|y||z|}\mu(x,(z\circ y))\\
&\overset{(\ref{eq:defi:H-anti-pre-Lie Poisson2})}{=}
2(-1)^{|x|(|y|+|z|)}\alpha(y)\circ\mu(z,x)-2(-1)^{|z|(|x|+|y|)}\alpha(z)\circ\mu(x,y)\\&-2\mu(x,(y\circ z))+2(-1)^{|y||z|}\mu(\alpha(x),(z\circ y))\\
&\overset{(\ref{eq:defi:H-anti-pre-Lie Poisson1})}{=}
2(-1)^{|x|(|y|+|z|)}\alpha(y)\circ\mu(z,x)-2(-1)^{|z|(|x|+|y|)}\alpha(z)\circ\mu(x,y)\\&+(-1)^{|x|(|y|+|z|)}\mu(\alpha(y),(z\circ x))-(-1)^{|x|(|y|+|z|)+|y||z|}\mu(\alpha(z),(y\circ x))\\
&\overset{(\ref{eq:defi:H-anti-pre-Lie Poisson3})}{=}
(-1)^{|x|(|y|+|z|)+|y||z|}\mu(z,y)\circ\alpha(x)+(-1)^{|x|(|y|+|z|)+|y||z|}\mu(\alpha(z),(y\circ x))\\&-(-1)^{|x|(|y|+|z|)}\mu(y,z)\circ \alpha(x)-(-1)^{|x|(|y|+|z|)}\mu(\alpha(y),(z\circ x))\\&+(-1)^{|x|(|y|+|z|)}\mu(\alpha(y),(z\circ x))-(-1)^{|x|(|y|+|z|)+|y||z|}\mu(\alpha(z),(y\circ x))\\&=0,
\end{align*}
\begin{eqnarray*}
&&2\mu((x\star y),\alpha(z))-2(-1)^{|x||y|}\mu((y\star x),\alpha(z))+2(-1)^{|x||y|}\alpha(y)\star\mu(x,z)-2\alpha(x)\star\mu(y,z)\\
&&=6\mu((x\circ y),\alpha(z))-6(-1)^{|x||y|}\mu((y\circ x),\alpha(z))+2(-1)^{|x||y|}\alpha(y)\circ\mu(x,z)-4(-1)^{|y||z|}\mu(x,z)\circ\alpha(y)\\&&\qquad-2\alpha(x)\circ\mu(y,z)+4(-1)^{|x|(|y|+|z|)}\mu(y,z)\circ\alpha(x)\\
&&\overset{(\ref{eq:defi:H-anti-pre-Lie
Poisson1}),(\ref{eq:defi:H-anti-pre-Lie Poisson3})}{=}
3(-1)^{|x||y|}\mu(\alpha(y),(x\circ z))-3\mu(\alpha(x),(y\circ z))+\mu(x,y)\circ\alpha(z)+\mu(\alpha(x),(y\circ z))\\&&\qquad-4(-1)^{|y||z|}\mu(x,z)\circ\alpha(y)-(-1)^{|x||y|}\mu(y,x)\circ\alpha(z)\\
&&\qquad-(-1)^{|x||y|}\mu(\alpha(y),(x\circ z))+4(-1)^{|x|(|y|+|z|)}\mu(y,z)\circ\alpha(x)\\
&&=2(-1)^{|x||y|}\mu(\alpha(y),(x\circ z))-2\mu(\alpha(x),(y\circ z))+4(-1)^{|x|(|y|+|z|)}\mu(y,z)\circ\alpha(x)-4(-1)^{|y||z|}\mu(x,z)\circ\alpha(y)\\
&&\overset{(\ref{eq:defi:H-anti-pre-Lie Poisson5})}{=}
4\alpha(x)\circ\mu(y,z)-4(-1)^{|x||y|}\alpha(y)\circ\mu(x,z)+4(-1)^{|x|(|y|+|z|)}\mu(y,z)\circ\alpha(x)-4(-1)^{|y||z|}\mu(x,z)\circ \alpha(y)\\&&\overset{(\ref{eq:defi:H-anti-pre-Lie Poisson8})}{=}0.
\end{eqnarray*}

Hence Eqs.~(\ref{eq:defi:H-Novikov Poisson superalgebra1}) and
(\ref{eq:defi:H-Novikov Poisson superalgebra2}) hold. Conversely, suppose
that Eqs.~(\ref{eq:defi:H-Novikov Poisson superalgebra1}) and
(\ref{eq:defi:H-Novikov Poisson superalgebra2}) hold. Then a similar
argument gives Eqs.~(\ref{eq:defi:H-anti-pre-Lie Poisson1}) and
(\ref{eq:defi:H-anti-pre-Lie Poisson2}). Indeed: for any $x,y,z\in\mathcal H(\mathcal A)$. We have
\begin{eqnarray*}
&&2(x\circ y)\cdot z-2(y\circ x)\cdot z+x\cdot(y\circ z)-y\cdot(x\circ z)\\
&&=-\frac{1}{3}(2(x\star y)\cdot z+4(y\star x)\cdot z-2(y\star x)\cdot z-4(x\star y)\cdot z+x\cdot(y\star z)+2x\cdot(z\star y)\\
&&\hspace{0.4cm}-y\cdot(x\star z)-2y\cdot(z\star x))\\
&&\overset{(\ref{eq:defi:H-Novikov Poisson superalgebra1})}{=}
-\frac{1}{3}(2(y\star x)\cdot z-2(x\star y)\cdot z+2x\cdot(z\star y)-2y\cdot(z\star x))\\
&&\overset{(\ref{eq:defi:H-Novikov Poisson superalgebra1})}{=}
-\frac{1}{3}(2(y\star x)\cdot z-2(x\star y)\cdot z+2z\cdot(x\star y)-2z\cdot(y\star x))=0,\\
&&2x\circ(y\cdot z)-(z\cdot x)\circ y-z\cdot(x\circ y)\\
&&=-\frac{1}{3}(2x\star(y\cdot z)+4(y\cdot z)\star x-(z\cdot x)\star y-2y\star(z\cdot x)-z\cdot(x\star y)-2z\cdot(y\star x))\\
&&\overset{(\ref{eq:defi:H-Novikov Poisson superalgebra1})}{=}
-\frac{1}{3}(2x\star(y\cdot z)+2z\cdot(y\star x)-2z\cdot(x\star y)-2y\star(z\cdot x))\overset{(\ref{eq:defi:H-Novikov Poisson superalgebra2})}{=}0.
\end{eqnarray*}
Hence Eqs.~(\ref{eq:defi:H-anti-pre-Lie Poisson1}) and
(\ref{eq:defi:H-anti-pre-Lie Poisson2}) hold.
\end{proof}



\begin{prop}
Let $(\mathcal P,\mathcal Q)$ be an admissible pair on a super-commutative Hom-associative
superalgebra $(\mathcal A,\mu,\alpha)$.
\begin{enumerate}
\item \label{it:1} Let $(\mathcal A,\star,\alpha)$ be the Hom-Novikov superalgebra given
by Eq.~(\ref{eq:ex:Novikov algebra from admissible pair}) in
Proposition \ref{ex:Novikov algebra from admissible pair}. Then
$(\mathcal A,\mu,\star,\alpha)$ is a Hom-Novikov-Poisson superalgebra.
\item \label{it:2} Let $(\mathcal A,\circ,\alpha)$ be the admissible Hom-Novikov
superalgebra given by Eq.~(\ref{eq:cor:admissible Novikov algebra from
admissible pair}) in Corollary \ref{cor:admissible Novikov algebra
from admissible pair}. Then $(\mathcal A,\mu,\circ,\alpha)$ is an admissible
Hom-Novikov-Poisson superalgebra.
\end{enumerate}
\end{prop}

\begin{proof}
(\ref{it:1}). Let $x,y,z\in\mathcal H(\mathcal A)$. Then we have
\begin{eqnarray*}
\mu(x,y)\star\alpha(z)=\mu(\mu(x,y),\mathcal Q(\alpha(z)))&=&\mu(\alpha(x),(y\star z)), \\
\alpha(x)\star\mu(y,z)-(-1)^{|x||y|}\alpha(y)\star\mu(x,z)&=&\mu(\alpha(x),\mathcal Q(\mu(y,z)))-(-1)^{|x||y|}\mu(\alpha(y),\mathcal Q(\mu(x,z)))\\&=&\mu(\mu(x,\mathcal Q(y)),\alpha(z))-(-1)^{|x||y|}\mu(\mu(y,\mathcal Q(x)),\alpha(z))\\&=&\mu((x\star y),\alpha(z))-(-1)^{|x||y|}\mu((y\star x),\alpha(z)).
\end{eqnarray*}
Hence Eqs.~(\ref{eq:defi:H-Novikov Poisson superalgebra1}) and
(\ref{eq:defi:H-Novikov Poisson superalgebra2}) hold.

(\ref{it:2}). It follows from Item~(\ref{it:1}) and Proposition \ref{pro:from admissible Novikov Poisson to Novikov Poisson}.
\end{proof}

\begin{exa}
Let $P$ be an even derivation on a super-commutative Hom-associative superalgebra
$(\mathcal A,\mu,\alpha)$. Define two bilinear operations $\star,\circ:\mathcal A\times\mathcal
A\rightarrow\mathcal A$ respectively by
\begin{equation*}
x\star y=\mu(x,P(y))+\mu(a,\mu(x,y)),\;\; \forall x,y\in\mathcal H(\mathcal A),
\end{equation*}
\begin{equation*}
x\circ y=\mu(x,P(y))+2\mu(P(x), y)+\mu(a,\mu(x,y)),\;\;\forall
x,y\in\mathcal H(\mathcal A), \end{equation*} where $a\in {\mathbb K}$ or $a\in \mathcal A$.
Then $(\mathcal A,\mu,\star,\alpha)$ is a Hom-Novikov-Poisson superalgebra and
$(\mathcal A,\mu,\circ,,\alpha)$ is an admissible Hom-Novikov-Poisson superalgebra.
\end{exa}


Next we consider the structure of the sub-adjacent Hom-Lie superalgebras of the Hom-anti-pre-Lie superalgebras in Hom-anti-pre-Lie Poisson superalgebras.

\begin{df}\label{defi:transposed Poisson algebra}
A \textbf{transposed Hom-Poisson superalgebra} is a quadruple
$(\mathcal A,\mu,[\cdot,\cdot],\alpha)$, where $(\mathcal A,\mu,\alpha)$ is a super-commutative
Hom-associative superalgebra, and $(\mathcal A,[\cdot,\cdot],\alpha)$ is a Hom-Lie superalgebra
satisfying the following equation:
\begin{equation}\label{eq:defi:transposed Poisson algebra}
2\mu(\alpha(z),[x,y])
=[\mu(z,x),\alpha(y)]
+(-1)^{|x||z|}[\alpha(x),\mu(z,y)],
\qquad \forall\, x,y,z\in\mathcal H(\mathcal A).
\end{equation}
\end{df}

\begin{exa}\label{Ex:from H-Novikov Poisson superalgebra to transposed H-Poisson superalgebra}
Let $(\mathcal A,\mu,\star,\alpha)$ be a Hom-Novikov-Poisson superalgebra and $(\mathcal A,[\cdot,\cdot],\alpha)$ be the sub-adjacent Hom-Lie superalgebra of $(\mathcal A,\star,\alpha)$. Then $(\mathcal A,\mu,[\cdot,\cdot],\alpha)$ is a transposed Hom-Poisson superalgebra.
\end{exa}

\begin{prop}\label{pro:sub-adjacent transposed H-Poisson superalgebra}
Let $(\mathcal A,\mu,\alpha)$ be a super-commutative Hom-associative superalgebra and
$(\mathcal A,\circ,\alpha)$ be a Hom-Lie-admissible superalgebra such that
Eqs.~(\ref{eq:defi:H-anti-pre-Lie Poisson1}) and
(\ref{eq:defi:H-anti-pre-Lie Poisson2}) hold.
Suppose that $(\mathcal A,[\cdot,\cdot],\alpha)$ is the
 sub-adjacent
Hom-Lie superalgebra of $(\mathcal A,\circ,\alpha)$.
Then $(\mathcal A,\mu,[\cdot,\cdot],\alpha)$ is a transposed Hom-Poisson superalgebra. 
In particular, if $(\mathcal A,\mu,\circ,\alpha)$ is a Hom-anti-pre-Lie Poisson
superalgebra, then $(\mathcal A,\mu,[\cdot,\cdot],\alpha)$ is a transposed Hom-Poisson superalgebra.
\end{prop}
\begin{proof}
Let $x,y,z\in\mathcal H(\mathcal A)$. Then we have
\begin{align*}
&\qquad[\mu(z,x),\alpha(y)]+(-1)^{|x||z|}[\alpha(x),\mu(z,y)]-2\mu(\alpha(z),[x,y])\\&=\mu(z,x)\circ\alpha(y)-(-1)^{|y|(|x|+|z|)}\alpha(y)\circ\mu(z,x)+(-1)^{|x||z|}\alpha(x)\circ\mu(z,y)-(-1)^{|x||y|}\mu(z,y)\circ\alpha(x)\\&\quad-2\mu(\alpha(z),[x,y])\\
&\overset{(\ref{eq:defi:H-anti-pre-Lie Poisson2})}{=}2(-1)^{|z|(|x|+|y|)}\alpha(x)\circ\mu(y,z)-\mu(\alpha(z),(x\circ y))-(-1)^{|y|(|x|+|z|)}\alpha(y)\circ\mu(z,x)\\&\quad+(-1)^{|x||z|}\alpha(x)\circ\mu(z,y)-2(-1)^{|z|(|x|+|y|)+|x||y|}\alpha(y)\circ\mu(x,z)\\&\quad+(-1)^{|x||y|}\mu(\alpha(z),(y\circ x))-2\mu(\alpha(z),(x\circ y))+2(-1)^{|x||y|}\mu(\alpha(z),(y\circ x))\\
&=3((-1)^{|z|(|x|+|y|)}\alpha(x)\circ\mu(y,z)-(-1)^{|y|(|x|+|z|)}\alpha(y)\circ \mu(z,x)-\mu(\alpha(z),(x\circ y))+(-1)^{|x||y|}\mu(\alpha(z),(y\circ x)))\\&\overset{(\ref{eq:defi:H-anti-pre-Lie Poisson9})}{=}0.
\end{align*}
Thus $(\mathcal A,\mu,[\cdot,\cdot],\alpha)$ is a transposed Hom-Poisson superalgebra.
\end{proof}

Conversely, in the nondegenerate case, we have the following conclusion.

\begin{prop}\label{thm:bilinear form on transposed Poisson algebra}
Let $(\mathcal A,\mu,[\cdot,\cdot],\alpha)$ be a transposed Hom-Poisson superalgebra. Suppose
there is a nondegenerate supersymmetric bilinear form $\mathfrak B$ on
$\mathcal A$, such that $\mathfrak B$ is invariant on $(\mathcal A,\mu,\alpha)$ and a
super-commutative $2$-cocycle on $(\mathcal A,[\cdot,\cdot],\alpha)$. Then  $(\mathcal A,\mu,\circ,\alpha)$ is
a Hom-anti-pre-Lie Poisson superalgebra, where the bilinear operation $\circ$
is given by Eq.\eqref{inv-Hom-anti-pre-Lie}.
\end{prop}

\begin{proof} By Theorem~\ref{ant-h-pr-Lie-from-nondeg-bil}, $(\mathcal A,\circ,\alpha)$ is a Hom-anti-pre-Lie superalgebra.
Let $x,y,z,t\in\mathcal H(\mathcal A)$. Then we have
\begin{eqnarray*}
&&\qquad\mathfrak B(2\mu([x,y],\alpha(z))+\mu(\alpha(x),(y\circ z))-(-1)^{|x||y|}\mu(\alpha(y),(x\circ z)),\alpha^2(t))\\&&=(-1)^{|z|(|x|+|y|)}\mathfrak B(\alpha(z),2\mu([x,y],t)+(-1)^{|x||y|}[y,\mu(x,t)]-[x,\mu(y,t)])\\&&=0,
\end{eqnarray*}
and
\begin{eqnarray*}
&&\qquad\mathfrak B(2\alpha(x)\circ\mu(y,z)-(-1)^{|z|(|x|+|y|)}\mu(z,x)\circ\alpha(y)-(-1)^{|z|(|x|+|y|)}\mu(\alpha(z),(x\circ y),\alpha^2(t))\\&&=(-1)^{|x||y|}\mathfrak B(\alpha(y),2(-1)^{|x||z|}\mu(z,[x,t])-(-1)^{|x||z|}[\mu(z,x),t]-[x,\mu(z,t)])\\&&=0.
\end{eqnarray*}
Thus $(\mathcal A,\mu,\circ,\alpha)$ is a Hom-anti-pre-Lie Poisson superalgebra.
\end{proof}

\section{Hom-anti-pre-Novikov superalgebras}

We introduce a notion of Hom-anti-pre-Novikov superalgebras as a new approach of splitting operations,
whose negative left and right multiplication operators compose the bimodules of the
associated Hom-Novikov superalgebras. The notions of anti-super-$\mathcal O$-operators on Hom-Novikov superalgebras are considered
 to interpret Hom-anti-pre-Novikov superalgebras. 
\subsection{Hom-anti-pre-Novikov superalgebras and Representations}

\begin{df}
A {\bf representation} of a Hom-Novikov superalgebra $(\mathcal A,\star,\alpha)$ is a quadruple $(V,\mathfrak l,\mathfrak r,\beta)$, 
where $V$ is a $\mathbb Z_2$-graded vector space, $\beta:V\to V$ is an even linear map and $\mathfrak l,\mathfrak r: \mathcal A\rightarrow \text{End}_{\mathbb K}(V)$ are even linear maps satisfying
\begin{align}
&\beta\mathfrak l(x)=\mathfrak l(\alpha(x))\beta,\;\;\beta\mathfrak r(x)=\mathfrak r(\alpha(x))\beta,\label{Nr0}\\
 &\mathfrak l(x\star y-(-1)^{|x||y|}y\star x)\beta=\mathfrak l(\alpha(x))\mathfrak l(y)-(-1)^{|x||y|}\mathfrak l(\alpha(y))\mathfrak l(x),\label{Nr1}\\
&\mathfrak l(\alpha(x))\mathfrak r(y)-(-1)^{|x||y|}\mathfrak r(\alpha(y))\mathfrak l(x)=\mathfrak r(x\star y)\beta-(-1)^{|x||y|}\mathfrak r(\alpha(y))\mathfrak r(x),\label{Nr2}\\
 &\mathfrak l(x\star y)\beta=(-1)^{|x||y|}\mathfrak r(\alpha(y))\mathfrak l(x),\ \ \
\mathfrak r(\alpha(x))\mathfrak r(y)=(-1)^{|x||y|}\mathfrak r(\alpha(y))\mathfrak r(x), \label{Nr3}
\end{align} for all $ x,y\in\mathcal H(\mathcal A).$
\end{df}

\begin{df}
 A {\bf Hom-anti-pre-Novikov superalgebra} is a $\mathbb Z_2$-graded vector space $\mathcal A$ together with two even blinear maps 
$\succ,\prec: \mathcal A\times\mathcal A \rightarrow\mathcal A$ and an even linear map $\alpha:\mathcal A\to\mathcal A$ satisfying
\begin{align}
&(x\circ y-(-1)^{|x||y|}y\circ x)\succ\alpha(z)=(-1)^{|x||y|}\alpha(y)\succ(x\succ z)-\alpha(x)\succ(y\succ z), \label{Aa1}\\
&\alpha(x)\prec(y\circ z)=(-1)^{|x||y|}(y\succ x) \prec\alpha(z)-(x\prec y)\prec\alpha(z)-(-1)^{|x||y|}\alpha(y)\succ (x\prec z),\label{Aa2}\\
&(x\circ y)\succ\alpha(z)=-(-1)^{|y||z|}(x\succ z)\prec\alpha(y),\label{Aa3}\\
&(x\prec y)\prec\alpha(z)=(-1)^{|y||z|}(x\prec z)\prec\alpha(y),\label{Aa4} \\
&(x\circ y-(-1)^{|x||y|}y\circ x)\prec\alpha(z)=\alpha(x)\succ(y\circ z)-(-1)^{|x||y|}\alpha(y)\succ(x\circ z),\label{Aa5}
\end{align}
for all $x,y,z\in\mathcal H(\mathcal A)$, where $x\circ y=x\succ y+x\prec y$.
\end{df}

\begin{exa}
Let $\mathcal{A} = \mathcal{A}_{\bar{0}} \oplus \mathcal{A}_{\bar{1}}$ be a $3$-dimensional $\mathbb{Z}_2$-graded vector space with basis $\{e_1, e_2, f\}$, where $|e_1| = |e_2| = \bar{0}$ and $|f| = \bar{1}$.

Define the non-zero products for the even bilinear operations $\succ, \prec: \mathcal{A} \times \mathcal{A} \to \mathcal{A}$ on basis elements by:
\begin{align*}
e_1 \succ e_1 &= e_2, & f \succ f &= e_2, \\
e_1 \prec e_1 &= -e_2, & f \prec f &= e_2,
\end{align*}
with all other products between basis elements vanishing. 

Define the linear map $\alpha: \mathcal{A} \to \mathcal{A}$ by:
\begin{equation*}
\alpha(e_1) = e_1, \quad \alpha(e_2) = e_2, \quad \alpha(f) = -f.
\end{equation*}

Then $(\mathcal{A}, \succ, \prec, \alpha)$ is a Hom-anti-pre-Novikov superalgebra.
\end{exa}

\begin{prop}
Let $\mathcal A$ be a $\mathbb Z_2$-graded vector space with two binary operations $\succ$ and $\prec$ and a linear map $\alpha: \mathcal A\to\mathcal A$. Define $x\circ y=x\succ y+x\prec y, \forall~x,y\in\mathcal H(\mathcal A)$.
Then the following conditions are equivalent:
\begin{enumerate}
\item $(\mathcal A,\succ,\prec,\alpha)$ is a Hom-anti-pre-Novikov superalgebra.
 \item  $(\mathcal A,\circ,\alpha)$ is a Hom-Novikov superalgebra
 and Eqs. (\ref{Aa1})-(\ref{Aa4}) hold for all $x,y,z\in\mathcal H(\mathcal A)$.
 \item $(\mathcal A,\circ,\alpha)$ is a Hom-Novikov superalgebra and $(\mathcal A,-\mathfrak L_{\succ},-\mathfrak R_{\prec},\alpha)$ is a representation of $(\mathcal A,\circ,\alpha)$,
\end{enumerate}
where  $\mathfrak L_{\succ},\mathfrak R_{\prec} : \mathcal A\longrightarrow \hbox{End} (\mathcal A)$ are the
linear maps defined by $\mathfrak L_{\succ}(x)(y)=x\succ y,~\mathfrak R_{\prec}(y)(x)=x\prec y$ 
for all $x,y\in\mathcal H(\mathcal A)$.
\end{prop}
\begin{proof}
$(1)\Longrightarrow (2)$
Suppose that $(\mathcal A,\succ,\prec,\alpha)$ is a Hom-anti-pre-Novikov superalgebra. Then, for all homogeneous  $x,y,z\in\mathcal A$, identities \eqref{Aa1}-\eqref{Aa4} hold 
\begin{align*}
 &(x\circ y)\circ \alpha(z)-\alpha(x)\circ(y\circ z)-(-1)^{|x||y|}(y\circ x)\circ \alpha(z)+(-1)^{|x||y|} \alpha(y)\circ (x\circ z) \\
=&(x\circ y-(-1)^{|x||y|}y\circ x)\succ \alpha(z)+(x\circ y-(-1)^{|x||y|} y\circ x)\prec \alpha(z)-\alpha(x)\succ (y\circ z) \\
&-\alpha(x)\prec (y\circ z)+(-1)^{|x||y|} \alpha(y)\succ (x\circ z)+(-1)^{|x||y|} \alpha(y)\prec (x\circ z)\\
\overset{(\ref{Aa1}),(\ref{Aa2})}{=}&(-1)^{|x||y|} \alpha(y)\succ (x\succ z)-\alpha(x)\succ (y\succ z)+(x\circ y-(-1)^{|x||y|} y\circ x)\prec \alpha((z) \\
&-\alpha(x)\succ(y\circ z)+(-1)^{|x||y|}\alpha(y)\succ(x\circ z)-(-1)^{|x||y|}(y\succ x)\prec \alpha(z)+(x\succ y)\prec \alpha(z)\\
&+(x\prec y)\prec \alpha(z)-(-1)^{|x||y|}(y\prec x)\prec \alpha(z)+(-1)^{|x||y|}\alpha(y)\succ (x\prec z)-\alpha(x)\succ (y\prec z)\\
=&(x\circ y-(-1)^{|x||y|} y\circ x)\prec \alpha(z)+(x\succ y +x\prec y)\prec \alpha(z)-(-1)^{|x||y|}(y\succ x+y\prec x)\prec \alpha(z) \\
&+(-1)^{|x||y|} \alpha(y)\succ (x\succ z+x\prec z)-\alpha(x)\succ (y\succ z+y\prec z)-\alpha(x)\succ (y\circ z)+(-1)^{|x||y|}\alpha(y)\succ (x\circ z)\\
=&2((x\circ y-(-1)^{|x||y|}y\circ x) \prec \alpha(z)-\alpha(x)\succ (y\circ z)+(-1)^{|x||y|} \alpha(y)\succ (x\circ z)) \\\overset{(\ref{Aa5})}{=}&0
\end{align*}
then 
\begin{align*}
     (x\circ y)\circ \alpha(z)-\alpha(x)\circ(y\circ z)=(-1)^{|x||y|}(y\circ x)\circ \alpha(z)-(-1)^{|x||y|} \alpha(y)\circ (x\circ z)
\end{align*}
  It follows immediately that, $(\mathcal A,\circ,\alpha)$ is a Hom-pre-Lie superalgebra.\\
In addition, a straightforward verification yields:
\begin{align*}
LHS &=(x\circ y)\circ\alpha(z)\\& =(x \circ y) \succ \alpha(z) + (x \succ y) \prec \alpha(z) + (x \prec y) \prec \alpha(z).
\end{align*}
Applying axioms \eqref{Aa3} and \eqref{Aa4}, we get
\begin{align}\label{eq:LHS-expanded}
\text{LHS} &= -(-1)^{|y||z|}(x \succ z) \prec \alpha(y) + (x \succ y) \prec \alpha(z) + (-1)^{|y||z|}(x \prec z) \prec \alpha(y).
\end{align}
Similarly, expanding the right-hand side yields
\begin{align}\label{eq:RHS-expanded}
\text{RHS} &= (-1)^{|y||z|}(x \circ z) \circ \alpha(y) \nonumber\\
&= -(x \succ y) \prec \alpha(z) + (-1)^{|y||z|}(x \succ z) \prec \alpha(y) + (-1)^{|y||z|}(x \prec z) \prec \alpha(y).
\end{align}
Subtracting \eqref{eq:RHS-expanded} from \eqref{eq:LHS-expanded}, the term $(-1)^{|y||z|}(x \prec z) \prec \alpha(y)$ cancels out, giving
\begin{align*}
\text{LHS} - \text{RHS} &= 2(x \succ y) \prec \alpha(z) - 2(-1)^{|y||z|}(x \succ z) \prec \alpha(y).
\end{align*}
Using the right-commutativity relation $(x \succ y) \prec \alpha(z) = (-1)^{|y||z|}(x \succ z) \prec \alpha(y)$, we conclude that $\text{LHS} - \text{RHS} = 0$ ,which gives that, the induced operation $"\circ"$ satisfies condition \eqref{eq:defi:Hom-Novikov superalgebras1}. Therefore, $(\mathcal A,\circ,\alpha)$ is a Hom-Novikov superalgebra.\\

$(2)\Longrightarrow (1)$ Follows immediately from a direct computation by using \eqref{eq:defi:Hom-Novikov superalgebras1} together with identities \eqref{Aa1}-\eqref{Aa4}.\\

$(1)\Longrightarrow (3)$ Suppose that $(\mathcal A,\succ,\prec,\alpha)$ is a Hom-anti-pre-Novikov superalgebra. From the first part of the proof, we already know that $(\mathcal A,\succ,\prec,\alpha)$ is a Hom-Novikov superalgebra. \\
It remains to show that $(\mathcal A,-\mathfrak L_{\succ},-\mathfrak R_{\prec},\alpha)$ is a representation of $(\mathcal A,\circ,\alpha)$.\\
For any $x,y,z\in\mathcal H(\mathcal A)$, we have:
\begin{align*}
 &\qquad-\mathfrak L_{\succ}(x\circ y-(-1)^{|x||y|}y\circ x) \alpha(z)-\mathfrak L_{\succ}(\alpha(x))\mathfrak L_{\succ}(y)z+(-1)^{|x||y|} \mathfrak L_{\succ}(\alpha(y))\mathfrak L_{\succ}(x)z \\
 &=(x\circ y-(-1)^{|x||y|}y\circ x)\succ\alpha(z))+\alpha(x)\succ(y\succ z)-(-1)^{|x||y|}\alpha(y)\succ(x\succ z) \\
 &\overset{(\ref{Aa1})}{=}0
\end{align*}
Similarly, we have
$$
 -\mathfrak L_{\succ}(x\circ y)\alpha(z)-(-1)^{|x||y|} \mathfrak R_{\prec}(\alpha(y))\mathfrak L_{\succ}(x)z 
 = (x\circ y)\succ\alpha(z)+(-1)^{|y||z|}(x\succ z)\prec\alpha(y)\overset{(\ref{Aa3})}{=}0,
$$

 \begin{align*}
(-1)^{|z|(|x|+|y|)} \mathfrak R_{\prec}(\alpha(z))\mathfrak R_{\prec}(x)y-(-1)^{|z|(|x|+|y|)} \mathfrak R_{\prec}(\alpha(y)) \mathfrak R_{\prec}(z)x
&=(x\prec y)\prec\alpha(z)-(-1)^{|y||z|}(x\prec z)\prec\alpha(y) \\
&\overset{(\ref{Aa4})}{=}0,
 \end{align*}
 and
\begin{align*}
&-(-1)^{|x|(|y|+|z|)}\mathfrak R_{\prec}(y\circ z)\alpha(x)-(-1)^{(|x||y|+|z|(|x|+|y|)}\mathfrak R_{\prec}(\alpha(z))\mathfrak L_{\succ}(y)x+(-1)^{(|x||y|+|z|(|x|+|y|)}\mathfrak R_{\prec}(\alpha(z))\mathfrak R_{\prec}(y)x\\&\qquad\qquad+(-1)^{|x|(|y|+|z|)}\mathfrak L_{\succ}(\alpha(y))\mathfrak R_{\prec}(z)x\\ 
=&\alpha(x)\prec(y\circ z)-(-1)^{|x||y|}(y\succ x) \prec\alpha(z)+(x\prec y)\prec\alpha(z)+(-1)^{|x||y|}\alpha(y)\succ (x\prec z) \\
\overset{(\ref{Aa2})}{=}&0.
\end{align*}
Therefore, conditions \eqref{Nr1}-\eqref{Nr3} are satisfied, which implies that, $(\mathcal A,-\mathfrak L_{\succ},-\mathfrak R_{\prec},\alpha)$ is a representation of $(\mathcal A,\circ,\alpha)$. \\

$(3)\Longrightarrow (1)$\\
Assume that $(\mathcal A,\circ,\alpha)$ is a Hom-Novikov superalgebra and $(\mathcal A,-\mathfrak L_{\succ},-\mathfrak R_{\prec},\alpha)$ is a representation of $(\mathcal A,\circ,\alpha)$. Then, for any $x,y,z\in\mathcal H(\mathcal A)$, we have \\
\begin{align*}
&\qquad(x\circ y-(-1)^{|x||y|}y\circ x)\succ\alpha(z))+\alpha(x)\succ(y\succ z)-(-1)^{|x||y|}\alpha(y)\succ(x\succ z)\\
&= -\mathfrak L_{\succ}(x\circ y-(-1)^{|x||y|}y\circ x) \alpha(z)-\mathfrak L_{\succ}(\alpha(x))\mathfrak L_{\succ}(y)z+(-1)^{|x||y|} \mathfrak L_{\succ}(\alpha(y))\mathfrak L_{\succ}(x)z \\
&=0.
\end{align*}
Similarly, we have
\begin{align*}
&\qquad\alpha(x)\prec(y\circ z)-(-1)^{|x||y|}(y\succ x) \prec\alpha(z)+(x\prec y)\prec\alpha(z)+(-1)^{|x||y|}\alpha(y)\succ (x\prec z) \\
&=-(-1)^{|x|(|y|+|z|)}\mathfrak R_{\prec}(y\circ z)\alpha(x)-(-1)^{(|x||y|+|z|(|x|+|y|)}\mathfrak R_{\prec}(\alpha(z))\mathfrak L_{\succ}(y)x+(-1)^{(|x||y|+|z|(|x|+|y|)}\mathfrak R_{\prec}(\alpha(z))\mathfrak R_{\prec}(y)x\\&\qquad+(-1)^{|x|(|y|+|z|)}\mathfrak L_{\succ}(\alpha(y))\mathfrak R_{\prec}(z)x\\ 
&=0,
\end{align*}
\begin{align*}
(x\circ y)\succ\alpha(z)+(-1)^{|y||z|}(x\succ z)\prec\alpha(y) 
&=-\mathfrak L_{\succ}(x\circ y)\alpha(z)-(-1)^{|x||y|} \mathfrak R_{\prec}(\alpha(y))\mathfrak L_{\succ}(x)z \\
&=0,  
\end{align*}
 \begin{align*}
&(x\prec y)\prec\alpha(z)-(-1)^{|y||z|}(x\prec z)\prec\alpha(y) 
&=(-1)^{|z|(|x|+|y|)} \mathfrak R_{\prec}(\alpha(z))\mathfrak R_{\prec}(x)y-(-1)^{|z|(|x|+|y|)} \mathfrak R_{\prec}(\alpha(y)) \mathfrak R_{\prec}(z)x
&=0,
\end{align*}
and
By the defining identity of a Hom-Novikov superalgebra, we have
\begin{align*}
 (x\circ y-(-1)^{|x||y|}y\circ x)\prec\alpha(z)=\alpha(x)\succ(y\circ z)-(-1)^{|x||y|}\alpha(y)\succ(x\circ z),
\end{align*}
Then, $(\mathcal A,\succ,\prec,\alpha)$ is a Hom-anti-pre-Novikov superalgebra.

\end{proof}

\subsection{ Anti-super-$\mathcal O$-operators }
\begin{df} Let $(V,\mathfrak l,\mathfrak r,\beta)$ be a representation of a Hom-Novikov superalgebra $(\mathcal A,\star,\alpha )$. An even linear map $T:V\rightarrow\mathcal A$ is 
called an \textbf{anti-super-$\mathcal{O}$-operator} on $(\mathcal A,\star,\alpha )$ associated to $(V,\mathfrak l,\mathfrak r,\beta)$ if $T$ satisfies
\begin{equation}\label{Ao10}
T(u)\star T(v)=-T(\mathfrak l(T(u))v+(-1)^{|u||v|}\mathfrak r(T(v))u),\ \ \ \forall~ u,v\in\mathcal H( V).
\end{equation}
Furthermore, $T$ is called \textbf{strong} if 
\begin{equation}\label{Ao1}
\mathfrak l(T(u)\star T(v)-(-1)^{|u||v|}T(v)\star T(u))\beta(w)+(-1)^{|v||w|}\mathfrak r(T(u)\circ T(w))\beta(v)-(-1)^{|u|(|v|+|w|)}r(T(v)\star T(w))\beta(u)=0,
\end{equation}
for all $u,v,w\in\mathcal H( V)$.
In particular, an anti-super-$\mathcal{O}$-operator $T$ on $(\mathcal A,\star,\alpha )$ associated with
the representation $(\mathcal A,\mathfrak L_{\star},\mathfrak R_{\star},\alpha)$ is called an anti-Rota-Baxter operator, that is, $T:\mathcal A\longrightarrow\mathcal A$
is a linear map satisfying 
\begin{equation}
T(x)\star T(y)=-T(T(x)\star y)+x\star T(y)),\ \ \ \forall~ x,y\in\mathcal H(\mathcal A). \end{equation}
An anti-Rota-Baxter operator $T$ is called \textbf{strong} if $T$ satisfies 
\begin{equation}
[T(x),T(y)]\star\alpha(z)+(-1)^{|x||y|}\alpha(y)\star (T(x)\star T(z))-\alpha(x)\star (T(y)\star T(z))=0,\ \ \ \forall~ x,y,z\in\mathcal H(\mathcal A),\end{equation}
where $[x,y]=x\star y-(-1)^{|x||y|}y\star x$.
\end{df}

\begin{prop}\label{Am1} Let $(V,\mathfrak l,\mathfrak r,\beta)$ be a representation of a Hom-Novikov sualgebra $(\mathcal A,\star,\alpha)$. Assume that
$T:V\longrightarrow\mathcal A$ is an anti-super-$\mathcal{O}$-operator on $(\mathcal A,\star,\alpha)$ associated to $(V,\mathfrak l,\mathfrak r,\beta)$. Define two
binary operations $\succ,\prec$ on $V$ respectively by
\begin{equation} \label{Ao2} u\succ v=-\mathfrak l(T(u))v, \ \ \ 
u\prec v=-(-1)^{|u||v|}\mathfrak r(T(v))u, \ \ \ \forall ~u,v\in\mathcal H(V).\end{equation} Then we have
\begin{enumerate}
\item For all $u,v,w\in\mathcal H(V)$,  denote $u\cdot v=u\succ v+u\prec v$, the following equations hold:
\begin{align}  ((-1)^{|u||v|}v\cdot u-u\cdot v)\succ\beta(w)&=\beta(u)\succ(v\succ w)-(-1)^{|u||v|}\beta(v)\succ(u\succ w),\label{Ao3} 
\\ (-1)^{|v||w|}\big((u\succ w)\prec\beta(v)-\beta(u)\succ(w\prec v)\big)&=(-1)^{|w|(|u|+|v|)}\big(\beta(w)\prec(u\cdot v)+(w\prec u)\prec\beta(v)\big),
\label{Ao4}
\\ (u\cdot v)\succ\beta(w)&=-(-1)^{|v||w|}(u\succ w)\prec\beta(v),\label{Ao5}\\ (-1)^{|w|(|u|+|v|)+|u||v|}(w\prec v)\prec\beta(u)&=(-1)^{|w|(|u|+|v|)}(w\prec u)\prec\beta(v).\nonumber\end{align}
\item $(V,\succ,\prec,\beta)$ is a Hom-anti-pre-Novikov superalgebra if and only if $T$ is strong. In this case, $T$ 
is a homomorphism of Hom-Novikov superalgebras from $(V,\cdot,\beta)$ to $(\mathcal A,\circ,\alpha)$. Moreover, there is an induced
Hom-anti-pre-Novikov superalgebra structure on $T(V)=\{T(u)|u\in V \} \subseteq A$ given by
\begin{equation*}T(u)\succ T(V)=T(u\succ v),\ \ \ T(u)\prec T(V)=T(u\prec v),\ \ ~\forall~u,v\in V
\end{equation*}
and $T$ 
is a homomorphism of Hom-anti-pre-Novikov superalgebras.
\end{enumerate}
\end{prop}

\begin{proof} ($1$)
Using Eqs. \eqref{Nr1} and \eqref{Ao2}, for all $u,v,w\in\mathcal H(V)$, we have
\begin{align*}& \beta(u)\succ(v\succ w)-(-1)^{|u||v|}\beta(v)\succ(u\succ w), 
\\&=\mathfrak l(T(\beta(u)))\mathfrak l(T(v))w-(-1)^{|u||v|}\mathfrak l(T(\beta(v)))\mathfrak l(T(u))w\\&=\mathfrak l(\alpha T(u))\mathfrak l(T(v))w-(-1)^{|u||v|}\mathfrak l(\alpha T(v))\mathfrak l(T(u))w\\&=\mathfrak l(T(u)\star  T(v)-(-1)^{|u||v|}T(v)\star T(u))\beta(w)
\\&=-\mathfrak l(T(\beta(l(T(u))v+(-1)^{|u||v|}\mathfrak r(T(v))u)-\mathfrak l(T(\mathfrak l(T(v))u+\mathfrak r(T(u))v)\beta(w)\\&=-(u\cdot v-(-1)^{|u||v|}v\cdot u)\succ\beta(w),
\end{align*}
which yields that Eq. \eqref{Ao3} holds. Take the same procedure, we can prove that Eqs. \eqref{Ao4}-\eqref{Ao5} hold.

($2$) In view of item ($1$), $(V,\succ,\prec,\beta)$ is a Hom-anti-pre-Novikov superalgebra if and only if
\begin{equation}\label{Ao6} (u\cdot v-(-1)^{|u||v|}v\cdot u)\prec\beta(w)=\beta(u)\succ(v\cdot w)-(-1)^{|u||v|}\beta(v)\succ(u\cdot w),\ \ \ \forall~u,v,w\in\mathcal H(V).\end{equation}
Thanks to Eqs. \eqref{Nr1}-\eqref{Nr2} and \eqref{Ao10}-\eqref{Ao2}, we get
\begin{align*}&(u\cdot v-(-1)^{|u||v|}v\cdot u)\prec\beta(w)-\beta(u)\succ(v\cdot w)+(-1)^{|u||v|}\beta(v)\succ(u\cdot w)
\\=&(-1)^{|w|(|u|+|v|)}\mathfrak r(T(\beta(w)))(\mathfrak l(T(u))v+(-1)^{|u||v|}\mathfrak r(T(v))u\\&-(-1)^{|w|(|u|+|v|)+|u||v|}\mathfrak r(T(\beta(w)))(\mathfrak l(T(v))u+\mathfrak r(T(u))v)\\&-\mathfrak l(T\beta(u))(\mathfrak l(Tv)w+(-1)^{|v||w|}\mathfrak r(T(w))v)+(-1)^{|u||v|}\mathfrak l(T\beta(v))(\mathfrak l(T(u))w+(-1)^{|u||w|}\mathfrak r(T(w))u)
\\=&(-1)^{|u|(|v|+|w|)}\mathfrak r(T(v)\star T(w))\beta(u)-(-1)^{|v||w|}\mathfrak r(T(u)\star T(w))\beta(v)\\&-\mathfrak l(T(u)\star T(v)-(-1)^{|u||v|}T(v)\star T(u))\beta(w).\end{align*}
Thus, Eq. (\ref{Ao6}) holds if and only if Eq. (\ref{Ao1}) holds.
The other conclusions follow immediately.

\end{proof}

\begin{lem}\label{inv-ant-R-B-Strong}
Any invertible anti-super-$\mathcal O$-operator on a Hom-Novikov superalgebra is automatically strong.    
\end{lem}
\begin{proof}
Let invertible anti-super-$\mathcal O$-operator on a Hom-Novikov superalgebra $(\mathcal A,\star,\alpha)$ with respect to a representation 
$(V,\mathfrak l,\mathfrak r,\beta)$. In view of Eqs. \eqref{Ao10} and \eqref{Ao2}, we have
\begin{align*}
&(T(u)\star T(v))\star\alpha(T(w))
=-T(\mathfrak l(T(u))v+(-1)^{|u||v|}\mathfrak r(T(v))u)\star\alpha(T(w))
\\=&T[\mathfrak l(T(\mathfrak l(T(u))v+(-1)^{|u||v|}\mathfrak r(T(v))u))\beta(w)+(-1)^{|w|(|u|+|v|)}\mathfrak r(T(\beta(w)))(\mathfrak l(T(u))v+(-1)^{|u||v|}\mathfrak r(T(v))u)]
\\=&T[(u\cdot v)\succ\beta(w)+(u\cdot v)\prec\beta(w)].
\end{align*}
Analogously,
\begin{align*} &(T(v)\star T(u))\star\alpha(T(w))=T[(v\cdot u)\succ\beta(w)+(v\cdot u)\prec\beta(w)],
\\& \alpha(T(u))\star (T(v)\star T(w))=\beta(u)\succ(v\cdot w)+\beta(u)\prec(v\cdot w),\\&\alpha(T(v))\star (T(u)\star T(w))=\beta(v)\succ(u\cdot w)+\beta(v)\prec(u\cdot w).\end{align*}
Since $(\mathcal A,\star,\alpha)$ is a Hom-Novikov superalgebra and $T$ is
invertible, we obtain
\begin{equation*}(u\cdot v)\cdot\beta(w)-(-1)^{|u||v|}(v\cdot u)\cdot\beta(w)=\beta(u)\cdot( v\cdot w)-(-1)^{|u||v|}\beta(v)\cdot( u\cdot w).\end{equation*}
Combining items ($1$) and ($2$) of Proposition \ref{Am1}, Eq.~\eqref{Ao6} holds, which implies that $T$ is strong.    
\end{proof}

\begin{exa}
Consider the 3-dimensional Hom-Novikov superalgebra $(\mathcal{A}, \star, \alpha)$ presented in Example \ref{ex-H-Nov-sup}. Define the even linear map $\mathcal{R}: \mathcal{A} \to \mathcal{A}$ on the basis elements $\{e_1, e_2, f\}$ by
\[
\mathcal{R}(e_1) = -2e_1, \quad \mathcal{R}(e_2) = e_2, \quad \mathcal{R}(f) = -2f.
\]
A direct computation verifies that $\mathcal{R}$ is an invertible anti-Rota-Baxter operator of weight zero on $(\mathcal{A}, \star, \alpha)$. Consequently, by Lemma~\ref{inv-ant-R-B-Strong}, $\mathcal{R}$ is strong. Applying Proposition~\ref{Am1}, the induced structure $(\mathcal{A}, \succ, \prec, \alpha)$, with operations defined for all homogeneous elements $x, y \in \mathcal{H}(\mathcal{A})$ by
\[
x \succ y = -\mathcal{R}(x) \star y \quad \text{and} \quad x \prec y = -x \star \mathcal{R}(y),
\]
forms a Hom-anti-pre-Novikov superalgebra.
\end{exa}

\begin{thm}\label{Am2}
Let $(\mathcal A,\star,\alpha)$ be a Hom-Novikov superalgebra. Then there is a
compatible Hom-anti-pre-Novikov superalgebra structure on $(\mathcal A,\star,\alpha)$ if and
only if there exists an invertible anti-super-$\mathcal{O}$-operator on
$(\mathcal A,\star,\alpha)$.
\end{thm}

\begin{proof}
Assume that $(\mathcal A,\succ,\prec,\alpha)$ is a compatible Hom-anti-pre-Novikov superalgebra structure on $(\mathcal A,\star,\alpha)$. Then
\begin{equation*}x\star
y=x\succ y+x\prec y=-(-\mathfrak L_{\succ}
(x) y-(-1)^{|x||y|}\mathfrak R_{\prec} (y)x),\;\;\forall x,y\in\mathcal H(\mathcal A),\end{equation*}
 which means that the
identity map ${\rm Id}:\mathcal A\rightarrow\mathcal A$ is an invertible
anti-super-$\mathcal{O}$-operator on $(\mathcal A,\star,\alpha)$ associated to the
representation $(\mathcal A,-\mathfrak L_{\succ},-\mathfrak R_{\prec},\alpha)$.

On the other hand, assume that $T:V\rightarrow \mathcal A$ is an invertible
anti-super-$\mathcal{O}$-operator on $(\mathcal A,\star,\alpha)$ associated to a
representation $(V, \mathfrak l,\mathfrak r,\beta)$ of $(\mathcal A,\star,\alpha)$. 
In view of Proposition \ref{Am1}, there exists a Hom-anti-pre-Novikov superalgebra
 structure on $V$ and $T(V)=\mathcal A$ given by
Eq.~\eqref{Ao2}. For all $x,y\in\mathcal A$, due to $T$ being invertible, there exist $u,v\in V$ such that $x=T(u),y=T(v)$. Thus
we have
\begin{align*}
x\star y&=T(u)\star T(v)=-T(\mathfrak l(T(u))v+(-1)^{|u||v|}\mathfrak r(T(v))u)=T(u\succ v+u\prec v)\\
&=T(u)\succ T(v)+T(u)\prec T(v)=x\succ y+x\prec y,
\end{align*}
which indicates that $(\mathcal A,\succ,\prec,\alpha)$ is a compatible Hom-anti-pre-Novikov superalgebra
 structure on $(\mathcal A,\star,\alpha)$.
\end{proof}

\end{document}